\documentclass[a4paper,11pt,reqno]{amsart}

\usepackage{amssymb}
\usepackage{amstext}
\usepackage{amsmath}
\usepackage{amscd}
\usepackage{amsthm}
\usepackage{amsfonts}
\usepackage{enumerate}
\usepackage{graphicx}
\usepackage{latexsym}
\usepackage{mathrsfs}

\usepackage{array}   
\newcolumntype{L}{>{$}l<{$}} 

\usepackage[all,color]{xy}

\xyoption{all}

\usepackage{tikz}
\pgfdeclarelayer{nodelayer}
\pgfdeclarelayer{edgelayer}
\pgfsetlayers{edgelayer,nodelayer,main}
\tikzstyle{obj}=[]
\tikzstyle{tri}=[circle, draw, fill=black!50, inner sep=0pt, minimum width=4pt]

\usepackage{pstricks}
\usepackage{lscape}

\newtheorem{theorem}{Theorem}[section]

\newtheorem{corollary}[theorem]{Corollary}
\newtheorem{lemma}[theorem]{Lemma}
\newtheorem{proposition}[theorem]{Proposition}
\newtheorem{definition-proposition}[theorem]{Definition-Proposition}

\newtheorem*{thma}{Theorem A}
\newtheorem*{thmb}{Theorem B}
\newtheorem*{thmc}{Theorem C}

\theoremstyle{definition}
\newtheorem{definition}[theorem]{Definition}
\newtheorem{remark}[theorem]{Remark}
\newtheorem{example}[theorem]{Example}

\DeclareMathOperator{\gldim}{gldim}

\newcommand{\add}{\operatorname{add}\nolimits}
\newcommand{\arr}[1]{\overset{#1}{\rightarrow}}
\newcommand{\arrr}[1]{\xrightarrow{#1}}
\newcommand{\Aut}{\operatorname{Aut}\nolimits}

\newcommand{\CC}{\mathscr{C}}

\newcommand{\Cok}{\operatorname{Cok}\nolimits}

\newcommand{\curlU}{\mathscr{U}}

\newcommand{\Db}{\operatorname{D^b}\nolimits}

\newcommand{\der}{\Db}
\newcommand{\dert}{\otimes^{\mathbb{L}}}

\newcommand{\en}{{\operatorname{en}\nolimits}}
\newcommand{\End}{\operatorname{End}\nolimits}
\newcommand{\ev}{\operatorname{ev}\nolimits}
\newcommand{\coev}{\operatorname{coev}\nolimits}
\newcommand{\ext}{\operatorname{ext}\nolimits}
\newcommand{\Ext}{\operatorname{Ext}\nolimits}
\newcommand{\F}{\mathbb{F}}

\newcommand{\fs}{\text{.}}
\newcommand{\gen}[1]{\langle#1\rangle}
\newcommand{\gl}{\mathop{{\rm gl.dim\,}}\nolimits}

\newcommand{\grsh}[1]{(#1)}
\newcommand{\hd}{\head}
\newcommand{\head}{\operatorname{head}\nolimits}

\newcommand{\Hom}{\operatorname{Hom}\nolimits}

\newcommand{\II}{\mathscr{I}}

\newcommand{\Image}{\operatorname{Im}\nolimits}
\newcommand{\into}{\hookrightarrow}
\newcommand{\Ker}{\operatorname{Ker}\nolimits}

\newcommand{\mMod}{\operatorname{-mod}}
\newcommand{\mModm}{\operatorname{-mod-}}

\newcommand{\Modm}{\operatorname{mod-}}
\newcommand{\mproj}{\operatorname{-proj}}

\newcommand{\onto}{\twoheadrightarrow}

\newcommand{\op}{{\operatorname{op}\nolimits}}
\newcommand{\pd}{\operatorname{proj.dim}\nolimits}

\newcommand{\PP}{\mathscr{P}}

\newcommand{\Q}{\mathbb{Q}}

\newcommand{\rad}{\operatorname{rad}\nolimits}
\newcommand{\RHom}{\mathbb{R}\strut\kern-.2em\operatorname{Hom}\nolimits}

\newcommand{\soc}{\operatorname{soc}\nolimits}
\newcommand{\st}{\;\left|\right.\;}
\newcommand{\stmod}{\operatorname{-\underline{mod}}}
\newcommand{\costmod}{\operatorname{-\overline{mod}}}

\newcommand{\Te}{\tens}
\newcommand{\tens}{\operatorname{T}\nolimits}
\newcommand{\Tens}{\tens}
\newcommand{\Tor}{\operatorname{Tor}\nolimits}
\newcommand{\Tr}{\operatorname{Tr}\nolimits}
\newcommand{\triv}{\operatorname{Triv}\nolimits}
\newcommand{\Triv}{\triv}

\newcommand{\UU}{\mathscr{U}}
\newcommand{\XX}{\mathscr{X}}
\newcommand{\YY}{\mathscr{Y}}
\newcommand{\Z}{\mathbb{Z}}
\renewcommand{\mod}{\operatorname{mod}\nolimits}

\newcommand{\scrX}{\mathscr{X}}

\newcommand{\grModm}{\operatorname{mod^\Z-}}
\newcommand{\mgrMod}{\operatorname{-mod^\Z}}
\newcommand{\bigrModm}{\operatorname{mod^{\Z^2}-}}
\newcommand{\mbigrMod}{\operatorname{-mod^{\Z^2}}}
\newcommand{\grprojm}{\operatorname{proj^\Z-}}
\newcommand{\mgrproj}{\operatorname{-proj^\Z}}
\newcommand{\Ggrprojm}{\operatorname{proj^\Psi-}}
\newcommand{\mGgrproj}{\operatorname{-proj^\Psi}}

\newcommand{\mGgrmod}{\operatorname{-mod^\Psi}}
\newcommand{\mbigrproj}{\operatorname{-proj^{\Z^2}}}
\newcommand{\bigrprojm}{\operatorname{proj^{\Z^2}-}}
\newcommand{\lModm}{\operatorname{Mod-}}
\newcommand{\mlMod}{\operatorname{-Mod}}
\newcommand{\grlModm}{\operatorname{Mod^\Z-}}
\newcommand{\mgrlMod}{\operatorname{-Mod^\Z}}

\begin{document}

\title{Higher preprojective algebras, Koszul algebras, and superpotentials}
\author{Joseph Grant}
\address{School of Mathematics, University of East Anglia,
Norwich, NR4 7TJ, United Kingdom}
\email{j.grant@uea.ac.uk}
\author{Osamu Iyama}
\address{Graduate School of Mathematics, Nagoya University, Chikusa-ku, Nagoya, 464-8602 Japan}
\email{iyama@math.nagoya-u.ac.jp}
\thanks{J.G. was supported first by the Japan Society for the Promotion of Science and then by the Engineering and Physical Sciences Research Council [grant number EP/G007497/1].  O.I. was supported by JSPS Grant-in-Aid for Scientific Research (B) 24340004, (B) 16H03923, (C) 18K03209 and (S) 15H05738.}


\begin{abstract}
In this article we study higher preprojective algebras, showing that various known results for ordinary preprojective algebras generalize to the higher setting.  We first show that the quiver of the higher preprojective algebra is obtained by adding arrows to the quiver of the original algebra, and these arrows can be read off from the last term of the bimodule resolution of the original algebra.  In the Koszul case we are able to obtain the new relations of the higher preprojective algebra by differentiating a superpotential and we show that when our original algebra is $d$-hereditary all the relations come from the superpotential.

We then construct projective resolutions of all simple modules for the higher preprojective algebra of a $d$-hereditary algebra.  This allows us to recover various known homological properties of the higher preprojective algebras and to obtain a large class of almost Koszul dual pairs of algebras.  We also show that when our original algebra is Koszul there is a natural map from the quadratic dual of the higher preprojective algebra to a graded trivial extension algebra.
\end{abstract}

\maketitle

\section{Introduction}

The preprojective algebras of quivers are important algebras which appear in various areas of mathematics, e.g.\ Cohen-Macaulay modules \cite{Au,GL}, Kleinian singularities \cite{cb-ex}, cluster algebras \cite{GLS}, quantum groups \cite{KS,Lu}, quiver varieties \cite{Na}.
They were first introduced by Gelfand and Ponomarev \cite{gp} (see also \cite{dr}) by explicit quivers with relations:
The algebra $\Pi$ of a quiver $Q$ is the path algebra $\F \overline{Q}$ of the double quiver $\overline{Q}$ of $Q$ modulo the ideal generated by $\sum_{x\in Q_1}(xx^*-x^*x)$.
Baer, Geigle, and Lenzing gave a more conceptual construction of $\Pi$ based on the representation theory of the quiver $Q$ \cite{bgl}: 
Their algebra is the direct sum of spaces $\Hom_{\Lambda}(\Lambda,\tau^{-\ell}(\Lambda))$ for the inverse Auslander-Reiten translate $\tau^-$, with an obvious multiplication.
The algebras of Gelfand-Ponomarev and Baer-Geigle-Lenzing are isomorphic, as shown in \cite{ri,cb-quiv}.

Preprojective algebras enjoy very nice homological properties.
They enjoy a certain 2-Calabi-Yau property \cite{cb-man}: If $Q$ is non-Dynkin, then $\Pi$ is a $2$-Calabi-Yau algebra in the sense of Ginzburg. If $Q$ is Dynkin, then $\Pi$ is a self-injective algebra and its stable category is 2-Calabi-Yau.
They also enjoy  a certain Koszul property: If $Q$ is non-Dynkin, then $\Pi$ is a Koszul algebra. If $Q$ is Dynkin, then $\Pi$ is twisted periodic of period $3$ \cite{rs}, and moreover it is an almost Koszul algebra in the sense of Brenner, Butler, and King \cite{bbk}.

Recently, an analogue of preprojective algebras was studied in cluster theory \cite{kel-dcy} and higher-dimensional Auslander-Reiten theory \cite{iya-higher-ar}.
For a finite dimensional algebra $\Lambda$ of global dimension $d$, its preprojective algebra is defined as $\Hom_{\Lambda}(\Lambda,\tau_d^{-\ell}(\Lambda))$, where $\tau_d$ and $\tau_d^-$ are higher analogues of the Auslander-Reiten translates.
This algebra is the 0-th cohomology of the $(d+1)$-Calabi-Yau completion \cite{kel-dcy}, which is a $(d+1)$-Calabi-Yau differential graded algebra.
When $\Lambda$ is a so-called $d$-hereditary algebra, its higher preprojective algebra enjoys nice homological properties, including the $(d+1)$-Calabi-Yau property \cite{AO,hi-selfinj,io-napr,io-stab,hio}.
Higher preprojective algebras also appear in conformal field theory \cite{ep-nak,ep2} and in commutative and non-commutative algebraic geometry \cite{bh,bhi,bs,himo,mi,mm} where they are non-commutative analogues of anticanonical bundles.

A natural question arises: can we describe these higher preprojective algebras by quivers and relations, generalizing the description of Gelfand and Ponomarev? This is important in practice since having a description by a quiver and relations often makes calculations much easier to perform.
When $\Lambda$ has global dimension exactly $2$, the higher preprojective algebra is isomorphic to the Jacobi algebra of a certain quiver with potential \cite{kel-dcy,hi-selfinj}, whose relations are given by taking formal partial differentials of the potential.
Quivers with potential appeared in physicists' study of mirror symmetry, and also played a key role in categorification of Fomin-Zelevinsky cluster algebras \cite{dwz}.

It is a difficult problem to give a description of the higher preprojective algebra of a general finite-dimensional algebra in terms of a quiver and relations. Here, we impose the restriction that $\Lambda$ should be a Koszul algebra, which ensures its homological algebra is easier to understand. Then we are able to describe the quivers of the higher preprojective algebras, and to show that the new relations in the higher preprojective algebra come from taking higher formal partial differentials of a superpotential (see Theorem \ref{thm:pi-as-alg-pot}).
If we further assume that $\Lambda$ is a $d$-hereditary algebra \cite{hio}, then all the relations come from higher differentials of the superpotential, as in the known cases $d=1,2$.

\begin{thma}[Corollary \ref{cor:kos-nher-prep}]
 If $\Lambda\cong\F Q/(R)$ is Koszul and $d$-hereditary then
$$\Pi\cong\frac{\F\overline{Q}}{(\partial_pW)}$$
where the quiver $\overline{Q}$ is a quiver obtained from $Q$ by adding new arrows, and the relations $\partial_pW$ are obtained by differentiating a certain superpotential $W$ with respect to length $d-1$ paths of $\overline{Q}$.
\end{thma}

\noindent
In fact, our Theorem \ref{thm:qp-vosnex} is more general since $\Lambda$ can be a factor algebra of the tensor algebra $\tens_S(V)$ for a separable $\F$-algebra $S$.
Higher Jacobi algebras have been considered previously in representation theory, notably in work of Van den Bergh \cite{vdb}, and Bocklandt, Schedler, and Wemyss \cite{bsw} (see also \cite{dv,ms}). 
In the $d$-representation infinite case, which makes up half of the dichotomy of $d$-hereditary algebras, we recover the description of Calabi-Yau Koszul algebras given in \cite{bsw}. In the case where $\Lambda$ is a basic Koszul $d$-representation-infinite algebra over an algebraically closed field of characteristic 0, this description was also given by Thibault \cite{thi}.
These previous works only deals with the case when $\Lambda$ is Morita equivalent to a quotient of the path algebra of a quiver, while our Theorem \ref{thm:qp-vosnex} is much more general since $\Lambda$ can be a factor algebra of the tensor algebra $\tens_S(V)$ for a separable $\F$-algebra $S$. We give definitions of superpotentials in $\tens_S(V)$ and the associated higher Jacobi algebras which work in this generality, by using the 0th Hochschild homology (Definitions \ref{define superpotential} and \ref{2}). This requires some technical machinery prepared in Section \ref{subsec:spots}.
Although other definitions of (ordinary) Jacobi algebras for $\tens_S(V)$ were given in \cite{Ng,LZ,BL}, our definition seems to be more convenient.

We also obtain homological information about higher preprojective algebras.  Under the assumption that $\Lambda$ is $d$-hereditary, we are able to describe the projective resolutions of all simple $\Pi$-modules using the higher Auslander-Reiten theory of $\Lambda$.
In fact, we show that they are induced from $d$-almost split sequences (see Theorem \ref{d-ass gives projective resolution}).  As applications, we have the following results.

\begin{thmb}[Corollaries \ref{cor:nrf-pi-twper} and \ref{cor:gldimri} and Theorem \ref{thm:kos-alm-kos}]
Let $\Lambda$ be a $d$-hereditary algebra and $\Pi$ the $(d+1)$-preprojective algebra of $\Lambda$.
\begin{enumerate}[\rm(a)]
 \item If $\Lambda$ is $d$-representation finite, then $\Pi$ is self-injective, and the simple $\Pi$-modules have periodic projective resolutions. If moreover $\Lambda$ is Koszul, then $\Pi$ is almost Koszul.
 \item If $\Lambda$ is $d$-representation infinite, then $\Pi$ has global dimension $d+1$ (c.f.\ Appendix A), and the $\Z$-graded simple $\Pi$-modules $S$ satisfy $\RHom_{\Pi}(S,\Pi)\cong S^*(1)[-d-1]$. If moreover $\Lambda$ is Koszul, then so is $\Pi$.
\end{enumerate}
\end{thmb}

As a corollary, we deduce that in the $d$-representation finite case $\Pi$ is twisted periodic of period $d+2$.  This recovers a result of Dugas \cite{dug} and is related to the stably Calabi-Yau property \cite{io-stab}. In the $d$-representation infinite case, we deduce that $\Pi$ is a generalized Artin-Schelter regular algebra of dimension $d+1$ and Gorenstein parameter $1$ in the sense of \cite{mv-serre,ms11,mm,rr} (see also \cite{as}).  This recovers a result of Minamoto and Mori \cite{mm}. Our results show that higher Auslander-Reiten theory is essential in the study of Artin-Schelter regular algebras.

Next we consider quadratic duals.  We show that, when $\Lambda$ is Koszul, there is a natural map from the quadratic dual of the preprojective algebra to a graded trivial extension algebra of the quadratic dual of $\Lambda$. Moreover we characterize when this map is surjective (respectively, an isomorphism) (see Theorem \ref{thm:canonical map}). In particular, we prove the following result.

\begin{thmc}[Theorems \ref{thm:canonical map} and \ref{cor:nhered-surj}]
Let $\Lambda=\tens_S(V)/(R)$ be a Koszul algebra of global dimension $d$ over a separable $\F$-algebra $S$.
\begin{enumerate}[\rm(a)]
\item There exists a morphism $\phi:\Pi^!\to\triv_{d+1}(\Lambda^!)$ of $\Z$-graded $\F$-algebras.
\item If $\Lambda$ is $d$-hereditary then $\phi$ is surjective, and in this case $\phi$ is an isomorphism if and only if $\triv_{d+1}(\Lambda^!)$ is quadratic.
\end{enumerate}
\end{thmc}

In the $d=1$ case where $\Lambda=\F Q$ for $Q$ any connected acyclic quiver, we show that the map is an isomorphism whenever the underlying graph of $Q$ is not of type $A_1$ or $A_2$.
We finish by applying our results to the type $A$ $d$-hereditary algebras $\Lambda^{(d,s)}$ \cite{io-napr} and use Theorem B to deduce that the type $A$ higher preprojective algebras are almost Koszul algebras with parameters $(s-1,d+1)$, thus obtaining examples of $(p,q)$-Koszul algebras for all $p,q\geq2$.

Note that a special case of Theorem C was independently obtained by Guo \cite[Theorem 5.3]{G}.
His result corresponds to the ``if'' part of our Theorem \ref{thm:canonical map}(c) under the assuption that $\Lambda^!$ is given by a quiver with relations and $\triv_{d+1}(\Lambda^!)$ is Koszul. 
Also Theorem C is closely related to \cite[Section 5]{H}.

\medskip\noindent
\textbf{Acknowledgements:} 
Early versions of these results were obtained when the first author was a JSPS Postdoctoral Fellow at Nagoya University during 2010 and 2011. Some results were presented at meetings in Newcastle University (2012), University of Cambridge (2015) and Isaac Newton Institute (2017). The authors thank them for supporting our project.
They also acknowledge the hospitality of Syracuse University, Isaac Newton Institute, and Czech Technical University in Prague.

\section{Preliminaries}

Let $\Lambda$ be a finite-dimensional algebra over a field $\F$.  By default, a $\Lambda$-module will mean a finitely generated left $\Lambda$-module, and we denote the category of such modules by $\Lambda\mMod$.  The corresponding category of right modules is denoted $\Modm\Lambda$.
If $\scrX$ is a set of left or right modules, we denote by $\add\scrX$ the additive subcategory of modules isomorphic to summands of sums of elements of $\scrX$. We sometimes write $\add M$ for $\add\{M\}$.
We denote by $gf$ the composition of morphisms $f:X\to Y$ and $g:Y\to Z$.

We denote the enveloping algebra $\Lambda\otimes_\F\Lambda^\op$ of $\Lambda$ by $\Lambda^\en$.  We will assume that $\F$ acts centrally on all bimodules, and then we can identify the category $\Lambda^\en\mMod$ of left $\Lambda^\en$-modules with the category $\Lambda\mModm\Lambda$ of $\Lambda^\en$-modules.  
We have a duality $(-)^*=\Hom_\F(-,\F):\Lambda\mMod\arr{\sim} \Modm\Lambda$ which sends left modules to right modules and vice versa.  It extends to a duality $\Lambda^\en\mMod\arr{\sim}\Lambda^\en\mMod$ of bimodules.  

\subsection{Tensor algebras}

Let $M$ be a $\Lambda^\en$-module.  Recall that the tensor algebra $\Te_\Lambda(M)$ of $M$ is the $\Z$-graded vector space
$$\Te_\Lambda(M)=\bigoplus_{i\geq0}M^i$$
where $M^i=M\otimes_\Lambda\cdots\otimes_\Lambda M$ is the tensor product of $i$ copies of $M$ so, in particular, $M^0=\Lambda$.  There is an obvious graded multiplication map $M^i\times M^j\to M^{i+j}$ which sends the pair $(\lambda_1\otimes \lambda_2\otimes\ldots \otimes\lambda_i,\lambda_{i+1}\otimes\ldots\lambda_{i+j})$ of standard basis vectors to the concatenated vector $\lambda_1\otimes \lambda_2\otimes\ldots\otimes\lambda_{i+j}$, and so $\Te_\Lambda(M)$ is a nonnegatively $\Z$-graded algebra.
For later use, we prepare the following basic observations, whose proofs are left to the reader.

\begin{lemma}\label{formula for Tens}
Let $M$ be a $\Lambda^\en$-module, $T:=\Tens_\Lambda(M)$, and $I$ an ideal of $\Lambda$.
\begin{enumerate}[\rm(a)]
\item For a $\Lambda^\en$-module $N$, we have $\Tens_T(T\otimes_\Lambda N\otimes_\Lambda T)\cong\Tens_\Lambda(M\oplus N)$.
\item For a $\Lambda^\en$-submodule $L$ of $M$, we have $\Tens_\Lambda(M)/(I+L)\cong\Tens_{\Lambda/I}(M/(IM+MI+L))$.
\end{enumerate}
\end{lemma}

Let $\Lambda$ be a basic $\F$-algebra with Jacobson radical $J$. We assume that $S$ is a semisimple subalgebra of $\Lambda$ such that $\Lambda=S\oplus J$. Then we can write $\Lambda=\Tens_S(V)/I$ for an $S^\en$-module $V$ and an ideal $I$ of $\Tens_S(V)$. If $I$ is a homogeneous ideal then $\Lambda$ inherits a grading from $\Tens_S(V)$.
Any such nonnegatively $\Z$-graded algebra $\Lambda$ has a minimal $\Z$-graded projective $\Lambda^\en$-module resolution
\[\cdots\arr{\delta_3}P_2\arr{\delta_2}P_1\arr{\delta_1}P_0\to0,\]
where each projective module $P_i$ is generated in degrees greater than or equal to $i$.
Immediately, we have the following property. 

\begin{lemma}\label{lem:e-gen-pos-deg2}
For any $i\ge0$, the $\Z$-graded $\Lambda^\en$-module $\Ext^i_{\Lambda^\en}(\Lambda,\Lambda^\en)$ is generated in degrees greater than or equal to $-i$.
\end{lemma}

We can write each projective $\Lambda^\en$-module in the form $\Lambda\otimes_S K\otimes_S \Lambda$ for some $\Z$-graded projective $S^\en$-module $K$, where we consider $S$ as a $\Z$-graded algebra concentrated in degree $0$; see \cite{bk}.  In particular, we write $P_i=\Lambda\otimes_S K_i\otimes_S \Lambda$ for $\Z$-graded projective $S^\en$-modules $K_i$, for $0\leq i\leq d$.

In general $K_i\cong\Tor^\Lambda_i(S,S)$, and explicit descriptions for these spaces are known.  For $m\geq0$,
$$\Tor^\Lambda_{2m}(S,S)\cong\frac{I^m\cap JI^{m-1}J}{JI^m+I^mJ}\;\text{ and }\Tor^\Lambda_{2m+1}(S,S)=\frac{JI^m\cap I^mJ}{I^{m+1}+JI^mJ}\fs$$
For more information and references, see the introduction to \cite{bk}.  For certain kinds of algebras there are nicer descriptions of these spaces: see Section \ref{subsec:kos} and the final chapters of \cite{bk}.

As well as our vector-space duality $(-)^*$, we have dualities
\begin{eqnarray*}
(-)^\vee:=\Hom_{S^\en}(-,S^\en):S^\en\mMod\arr{\sim} S^\en\mMod,\\
(-)^{*\ell}:=\Hom_S(-,S):S^\en\mMod\arr{\sim} S^\en\mMod,\\
(-)^{*r}:=\Hom_{S^\op}(-,S):S^\en\mMod\arr{\sim} S^\en\mMod.
\end{eqnarray*}
For $S^\en$-modules $X$ and $Y$, we have functorial isomorphisms
\begin{eqnarray}\label{XY1}
Y^{*\ell}\otimes_SX^{*}\cong(X\otimes_SY)^{*}\\ \label{XY2}
Y^{*\ell}\otimes_SX^{*\ell}\cong(X\otimes_SY)^{*\ell}
\end{eqnarray}
sending $f\otimes g$ to $(x\otimes y\mapsto g(xf(y))$,
\begin{eqnarray}\label{XY3}
Y^{*r}\otimes_SX^{*r}\cong(X\otimes_SY)^{*r}
\end{eqnarray}
sending $f\otimes g$ to $(x\otimes y\mapsto f(g(x)y))$, and
\begin{eqnarray}\label{XY4}
Y^{*r}\otimes_SX^{\vee}\cong(X\otimes_SY)^{\vee}
\end{eqnarray}
sending $f\otimes g$ to $(x\otimes y\mapsto\sum_is_i\otimes f(s'_iy))$ for $g(x)=\sum_is_i\otimes s'_i$. 
For example, \eqref{XY4} can be checked as follows:
Since $\Hom_{S^\op}(Y,S^\en)\cong S\otimes_\F Y^{*r}\in S^\en\mproj$, we have
\[(X\otimes_SY)^\vee\cong\Hom_{S^\en}(X,\Hom_{S^\op}(Y,S^\en))\cong(S\otimes_\F Y^{*r})\otimes_{S^\en}X^\vee\cong Y^{*r}\otimes_SX^\vee.\]

Note the following simple lemma:
\begin{lemma}\label{lem-tensorhomiso}
Let $L$ be a $\Lambda\otimes_\F S^\op$-module, $X$ be a projective $S^\en$-module, and $M$ be a $S\otimes_\F \Lambda^\op$-module.
Then there is an isomorphism of $\Lambda^\en$-modules which is natural in $L$, $X$, and $M$:
$$\Hom_{\Lambda^\en}(L\otimes_S X\otimes_SM,\Lambda^\en)\cong\Hom_\Lambda(M,\Lambda)\otimes_S X^\vee\otimes_S\Hom_{\Lambda^\op}(L,\Lambda).$$
In particular, for any projective $S^\en$-module $X$, there is a functorial isomorphism of $\Lambda^\en$-modules
$$\Hom_{\Lambda^\en}(\Lambda\otimes_S X\otimes_S\Lambda,\Lambda^\en)\cong\Lambda\otimes_S X^\vee\otimes_S\Lambda.$$
\end{lemma}

\begin{proof}
We include a complete proof for the convenience of the reader.
Using the tensor-hom adjunctions, for any $X\in S^\en\mMod$ we have isomorphisms of $\Lambda^\en$-modules
\begin{eqnarray*}
&&\Hom_{\Lambda^\en}(L\otimes_S X\otimes_SM,\Lambda^\en)
\cong\Hom_{\Lambda^\en}((L\otimes_\F M)\otimes_{S^\en}X,\Lambda^\en)\\
&\cong&\Hom_{S^\en}(X,\Hom_{\Lambda^\en}(L\otimes_\F M,\Lambda^\en))
\cong X^\vee\otimes_{S^\en}\Hom_{\Lambda^\en}(L\otimes_\F M,\Lambda^\en)\\
&\cong&X^\vee\otimes_{S^\en}(\Hom_\Lambda(L,\Lambda)\otimes_\F\Hom_{\Lambda^\op}(M,\Lambda))\cong\Hom_\Lambda(M,\Lambda)\otimes_S X^\vee\otimes_S\Hom_{\Lambda^\op}(L,\Lambda).
\end{eqnarray*}
All our isomorphisms are natural.
\end{proof}

The four duals $(-)^*$, $(-)^{*\ell}$, $(-)^{*r}$ and $(-)^\vee$ are isomorphic to each other (e.g.\ \cite[Section 3]{ric02}, \cite[Section 2.1]{bsw}).
In fact, since $S$ is a symmetric $\F$-algebra, there exists an $\F$-linear form $t:S\to\F$ such that $t(xy)=t(yx)$ and the map $S\to S^*$ sending $x$ to $(y\mapsto t(xy))$ is an isomorphism. This gives isomorphisms
\[\alpha:=t\circ(-):(-)^{*\ell}\cong(-)^*,\ 
\beta:=t\circ(-):(-)^{*r}\cong(-)^*\ \mbox{ and }
\gamma:=(t\otimes1)\circ(-):(-)^\vee\cong(-)^{*r}\]
of functors.

For later use in Section 5, we now show that these isomorphisms are compatible with module structures in the following sense: 
Let $L=\bigoplus_{i\in\Z}L_i$ be a $\Z$-graded $\tens_S(V^{*\ell})^{\op}$-module, and let
\[L^{(*)}=\bigoplus_{i\in\Z}L_{-i}^*,\ L^{(*\ell)}=\bigoplus_{i\in\Z}L_{-i}^{*\ell},\ L^{(*r)}=\bigoplus_{i\in\Z}L_{-i}^{*r}\ \mbox{ and }\ L^{(\vee)}=\bigoplus_{i\in\Z}L_{-i}^\vee.\]
Then $L^{(*)}$ and $L^{(*\ell)}$ are $\Z$-graded $\tens_S(V^{*\ell})$-modules, and $L^{(*r)}$ and $L^{(\vee)}$ are $\Z$-graded $\tens_S(V^{*r})$-modules as follows:
The action of $\tens_S(V^{*\ell})^{\op}$ on $L$ is given by a morphism $a_i:L_i\otimes_SV^{*\ell}\to L_{i+1}$ of $S^\en$-modules for $i\in\Z$.
This corresponds to a morphism $b_i:L_i\to L_{i+1}\otimes_SV$ of $S^\en$-modules via Hom-tensor adjunction $\Hom_{S^\en}(A\otimes_SB,C)\cong\Hom_{S^\en}(A,\Hom_{S^\op}(B,C))$.
Applying $(-)^\dagger$ for $\dagger=*,*\ell,*r,\vee$, we obtain morphisms
\begin{eqnarray*}
V^{*\ell}\otimes_SL_{i+1}^\dagger\stackrel{\eqref{XY1}\eqref{XY2}}{\cong}(L_{i+1}\otimes_SV)^\dagger\xrightarrow{b_i^\dagger} L_i^\dagger\mbox{ for $\dagger=*$ or $*\ell$,}\\
V^{*r}\otimes_SL_{i+1}^\dagger\stackrel{\eqref{XY3}\eqref{XY4}}{\cong}(L_{i+1}\otimes_SV)^\dagger\xrightarrow{b_i^\dagger} L_i^\dagger\mbox{ for $\dagger=*r$ or $\vee$}
\end{eqnarray*}
of $S^\en$-modules, which give the desired structures on $L^{(*)}$, $L^{(*\ell)}$, $L^{(*r)}$ and $L^{(\vee)}$.

\begin{lemma}\label{duals}
\begin{enumerate}[\rm(a)]
\item We have isomorphisms
$L^{(*)}\cong L^{(*\ell)}$ of $\Z$-graded $\tens_S(V^{*\ell})$-modules, and $L^{(*r)}\cong L^{(\vee)}$ of $\Z$-graded $\tens_S(V^{*r})$-modules.
\item Under the isomorphism $\tens_S(V^{*r})\cong\tens_S(V^{*\ell})$ of algebras given by $\alpha_V^{-1}\beta_V:V^{*r}\cong V^{*\ell}$, we have isomorphisms $L^{(*)}\cong L^{(*\ell)}\cong L^{(*r)}\cong L^{(\vee)}$ of $\Z$-graded $\tens_S(V^{*r})$-modules.
\end{enumerate}
\end{lemma}

\begin{proof}
The assertions follow from the following commutative diagram.
\[\xymatrix@R1.5em@C4em{
V^{*r}\otimes_SL_{i+1}^{\vee}\ar[r]^(.45){\eqref{XY4}}\ar[d]_{1_{V^{*r}}\otimes\gamma_{L_{i+1}}}&(L_{i+1}\otimes_SV)^\vee\ar[d]_{\gamma_{L_{i+1}\otimes V}}\ar[r]^(.6){b_i^\vee}&L_i^\vee\ar[d]^{\gamma_{L_i}}\\
V^{*r}\otimes_SL_{i+1}^{*r}\ar[r]^(.45){\eqref{XY3}}\ar[d]_{(\alpha_V^{-1}\beta_V)\otimes\beta_{L_{i+1}}}&(L_{i+1}\otimes_SV)^{*r}\ar[d]_{\beta_{V^{*\ell}\otimes L_{i+1}}}\ar[r]^(.6){b_i^{*r}}&L_i^{*r}\ar[d]^{\beta_{L_i}}\\
V^{*\ell}\otimes_SL_{i+1}^{*}\ar[r]^(.45){\eqref{XY1}}&(L_{i+1}\otimes_SV)^*\ar[r]^(.6){b_i^*}&L_i^*\\
V^{*\ell}\otimes_SL_{i+1}^{*\ell}\ar[r]^(.45){\eqref{XY2}}\ar[u]^{1_{V^{*\ell}}\otimes\alpha_{L_{i+1}}}&(L_{i+1}\otimes_SV)^{*\ell}\ar[u]^{\alpha_{L_{i+1}\otimes V}}\ar[r]^(.6){b_i^{*\ell}}&L_i^{*\ell}\ar[u]_{\alpha_{L_i}}
}
\]
The right squares commute since $\alpha,\beta,\gamma$ are morphisms of functors. The left top square commutes since both the north-west composition and the south-west composition send $f\otimes g\in V^{*r}\otimes_SL_{i+1}^{\vee}$ to $(L_{i+1}\otimes_SV\ni x\otimes v\mapsto\sum_it(s_i)f(s'_iv)\in S)$, where $g(x)=\sum_is_i\otimes s'_i$.
The left bottom square also commutes since both the north-west composition and the south-west composition send $f\otimes g\in V^{*\ell}\otimes_SL_{i+1}^{*\ell}$ to $(L_{i+1}\otimes_SV\ni x\otimes v\mapsto t(g(xf(v)))\in\F)$.

To check that the left middle square commute, fix $f\otimes g\in V^{*r}\otimes_SL_{i+1}^{*r}$. The north-west composition sends $f\otimes g$ to $(x\otimes v\mapsto t(f(g(x)v)))$. 
The south-west composition sends $f\otimes g$ to $(x\otimes v\mapsto t(g(xf'(v))))$, where $f'=\alpha_V^{-1}\beta_V(f)\in V^{*\ell}$ satisfies $t\circ f=t\circ f'$.
These two elements coincide since $t(f(g(x)v)))=t(f'(g(x)v))=t(g(x)f'(v))=t(g(xf'(v)))$.
\end{proof}

\subsection{Graded algebras and Koszul algebras}\label{subsec:kos}

In this section, we give preliminaries on Koszul algebras, which were introduced in \cite{pri} and studied extensively in \cite{bgs}.

Let $\Lambda=\bigoplus_{i\geq0}\Lambda_i$ be a positively $\Z$-graded $\F$-algebra satisfying the following conditions:
\begin{enumerate}[$\bullet$]
\item $S:=\Lambda_0$ is a finite dimensional semisimple $\F$-algebra, or equivalently, the $\Z$-graded radical of $\Lambda$ coincides with $\Lambda_{>0}:=\bigoplus_{i>0}\Lambda_i$.
\item $\Lambda$ is generated in degree $1$, i.e., the multiplication map $\Lambda_1\otimes_\F \Lambda_1\otimes_\F \cdots\otimes_\F \Lambda_1\to\Lambda_j$ is surjective for each $j$.
\end{enumerate}
In this case, we call the grading a \emph{radical grading}. 

We assume, for simplicity, that $\Lambda$ is basic.
Our assumptions ensure that $\Lambda$ is a quotient of the tensor algebra $\Tens_S(V)$ where $V$ is the $S^\en$-module $\Lambda_1$.  When $\Lambda$ is finite-dimensional and $\F$ is algebraically closed, we can identify $S$ with the space $\F Q_0$ of vertices, and $V$ with the space $\F Q_1$ of arrows, of the Gabriel quiver $Q$ of $\Lambda$.

For a $\Z$-graded $\Lambda$-module $M$ and $j\in\Z$, let $M\grsh{j}$ denote the shifted $\Z$-graded $\Lambda$-module where $M\grsh{j}_i = M_{i+j}$.
A complex
$$\cdots\to M_1\to M_0\to M_{-1}\to\cdots$$
of $\Z$-graded $\Lambda$-modules is \emph{linear} if each module $M_i$ is generated in degree $i$ and each map is homogeneous of degree $0$.
The algebra $\Lambda$ is \emph{Koszul} if each simple module $S_i$ has a linear projective resolution.

All Koszul algebras are \emph{quadratic} in the sense that they can be written as a quotient of a tensor algebra:
$$\Lambda\cong\Te_S(V)/(R)$$
where $V$ is an $S^\en$-module, $R$ is a subset of $V\otimes_SV$, and $(R)$ is the ideal in $\Te_S(V)$ that it generates. 
To simplify the proofs, we will sometimes assume that $R$ is a sub-$S^\en$-module of $V\otimes_SV$ instead of just a subset.  In particular, it is a vector subspace.  This is no real restriction.

We view $S$ as a $\Z$-graded $\F$-algebra concentrated in degree $0$, and $V$ as a $\Z$-graded $S^\en$-module concentrated in degree $1$.  Then the tensor grading and the grading coming from $V$ coincide, and so we can safely refer to just the grading on $\Lambda$.

We record a useful lemma on quadratic algebras which can be checked easily:
\begin{lemma}\label{lem-quad012}
Let $\phi:\Tens_S(V)/(R)\to\Tens_{S'}(V')/(R')$ be a morphism of $\Z$-graded quadratic $\F$-algebras.
If $\phi$ is an isomorphism in degrees $0$, $1$, and $2$, then it is an isomorphism.
\end{lemma}

In the rest of this subsection, let $\Lambda\cong\Te_S(V)/(R)$ be a quadratic algebra.
We have $S^\en$-modules $K_0=S$, $K_1=V$, $K_2=R$, and $$K_j=(V\otimes_SK_{j-1})\cap(K_{j-1}\otimes_SV)=\bigcap_{i=0}^{j-2}V^i\otimes_SR\otimes_SV^{j-2-i}$$
for $j\geq3$.
Here, $V^i$ denotes the $i$-th tensor power $V\otimes_S\cdots\otimes_SV$.
Note that $K_i$ is concentrated in degree $i$ and that for $i<0$, we set $K_i=0$.  

Recall \cite[Section 2.7]{bgs} that if $U$ is a subset of $V^{i}$, the right orthogonal complement of $U$ is $U^\perp=\{f\in (V^{*\ell})^i\st f(U)=0\}$, where we identify $(V^{*\ell})^i$ with $(V^i)^{*\ell}$ by \eqref{XY2}.
The \emph{quadratic dual} of a quadratic algebra $\Lambda=\Te_S(V)/(R)$ is
$$\Lambda^!=\Te_S(V^{*\ell})/(R^\perp).$$
It is again quadratic.  If moreover $\Lambda$ is Koszul, then $\Lambda^!$ is also Koszul and it coincides with the opposite ext algebra $\left(\bigoplus_{i\geq0}\Ext^i_\Lambda(S,S)\right)^\op$ \cite[Proposition 2.10.1]{bgs}.
In this case, $\Lambda^!$ has the following description.

\begin{lemma}[{\cite[Section 2.8]{bgs}}]\label{Koszul dual is K*}
For a Koszul algebra $\Lambda$, we have an isomorphism of $\Z$-graded algebras
 $$\Lambda^!=\bigoplus_{i\geq 0}(\Lambda^!)_i\cong\bigoplus_{i\geq0}K_i^{*\ell},$$
where the $\Z$-graded algebra structure on $\bigoplus_{i\geq0}K_i^{*\ell}$ is given by the duals of $(\iota^\ell_i)^{*\ell}$ and $(\iota^r_i)^{*\ell}$.
\end{lemma}

Now we assume that $S$ is a separable $\F$-algebra
i.e.\ $S\otimes_\F\F'$ is semisimple for all field extensions $\F\subset\F'$, or equivalently, $S^\en$ is semisimple \cite[Theorem 9.2.11]{W}. Let
\[P_i=\Lambda\otimes_SK_i\otimes_S\Lambda=\Lambda^\en\otimes_{S^\en}K_i.\]
This is a projective $\Lambda^\en$-module since $S^\en$ is semisimple by our assumption.
We have obvious inclusions $\iota^\ell_i:K_i\into V\otimes_SK_{i-1}$ 
and $\iota^r_i:K_i\into K_{i-1}\otimes_SV$ 
of $S^\en$-modules and, combined with the multiplication for $\Lambda$, they induce maps 
$\hat{\iota}^\ell_i,\hat{\iota}^r_i:P_i\to P_{i-1}$.  Let
\[\delta_i=\hat{\iota}^\ell_i+(-1)^i\hat{\iota}^r_i.\]
One can check that these maps give a chain complex
\begin{equation}\label{Koszul bimodule complex}
\cdots \xrightarrow{\delta_3}P_2\xrightarrow{\delta_2}P_1\xrightarrow{\delta_1}P_0\to0
\end{equation}
which is called the \emph{Koszul bimodule complex}.  Note that, as $K_i\subseteq V^i$ and $V$ is concentrated in degree $1$, 
each $P_i$ is generated in degree $i$, i.e., the complex is linear.

The next result is an important characterization of Koszul algebras.  It can be found as, for example, \cite[Proposition A.2]{bg} and \cite[Theorem 9.2]{bk}.
\begin{theorem}\label{characterize koszul}
$\Lambda$ is Koszul if and only if the Koszul bimodule complex \eqref{Koszul bimodule complex} is its minimal projective resolution as a $\Lambda^\en$-module. 
\end{theorem}

In this paper, we need separability of $S$ when we consider bimodule resolutions including the Koszul bimodule resolutions.
We assume separability in Theorems C and \ref{characterize koszul}, Sections \ref{subsec:graded} and \ref{subsec:hpka}, Corollary \ref{cor:kos-nher-prep}, and Section \ref{ss:gtea}.

\subsection{Higher preprojective algebras}

Let
\[\tau=(-)^*\circ\Tr:\Lambda\stmod\to\Lambda\costmod\]
denote the Auslander-Reiten translation, which is a functor from the stable category of modules over $\Lambda$ to the costable category, and
\[\tau^-=\Tr\circ(-)^*:\Lambda\costmod\to\Lambda\stmod\]
the inverse Auslander-Reiten translation.  Note that if $\Lambda$ is hereditary then the Auslander-Reiten translation in fact can be regarded as an endofunctor of the module category: 
$\tau = \Ext^1(-,\Lambda)^*$ and $\tau^- = \Ext^1_\Lambda(\Lambda^*,-)$.
Moreover, $\tau^-$ is left adjoint to $\tau$.

Recall \cite{iya-higher-ar} that the $d$-Auslander-Reiten translation and inverse $d$-Auslander-Reiten translation are defined as
\[\tau_d=\tau\Omega^{d-1}:\Lambda\stmod\to\Lambda\costmod\ \mbox{ and }\ \tau_d^-=\tau^-\Omega^{-(d-1)}:\Lambda\costmod\to\Lambda\stmod,\]
where $\Omega:\Lambda\stmod\to\Lambda\stmod$ denotes the syzygy functor and $\Omega^-:\Lambda\costmod\to\Lambda\costmod$ the cosyzygy functor. 
If $\gldim\Lambda\leq d$ then we regard $\tau_d$ and $\tau_d^-$ as the endofunctors
\[\tau_d=\Ext^d_\Lambda(-,\Lambda)^*:\Lambda\mMod\to\Lambda\mMod\ \mbox{ and }\ \tau_d^-=\Ext^d_\Lambda(\Lambda^*,-):\Lambda\mMod\to\Lambda\mMod\]
of $\Lambda\mMod$.

Generalizing the classical case, we have two distinguished classes of modules.
\begin{definition}[{\cite[Definition 4.7]{hio}}]\label{def:prep-prei2}
We have the following two full subcategories $\PP$ and $\II$ of $\Lambda\mMod$:
\[\PP:=\add\{\tau_d^{-i}(\Lambda)\ |\ i\ge0\}\ \mbox{ and }\ \II:=\add\{\tau_d^{i}(\Lambda^*)\ |\ i\ge0\}.\]
Any module in $\PP$ is called \emph{$d$-preprojective}, and any module in $\II$ is called \emph{$d$-preinjective}.
\end{definition}

In the rest of this section, we assume that $\Lambda$ has global dimension $d$. The $\Lambda^\en$-module
\[E:=\Ext^d_\Lambda(\Lambda^*,\Lambda)\]
plays a central role in this paper.
We take this opportunity to record a useful lemma, which makes the $\Lambda^\en$-module structure of $E$ clearer.
\begin{lemma}\label{lem:alt-E}
We have isomorphisms
\[E\cong\Ext^d_{\Lambda^\en}(\Lambda,\Lambda^\en)\cong\Ext^d_{\Lambda^\op}(\Lambda^*,\Lambda)\]
of $\Lambda^\en$-modules.
\end{lemma}

\begin{proof}
For each finite-dimensional $\Lambda$-module $M$, there is a natural isomorphism $M\cong M^{**}$.  Then we use the natural isomorphism of finite-dimensional vector spaces $V^*\otimes_\F W\cong\Hom_\F(V,W)$ to see that we have an isomorphism of $\Lambda^\en$-modules
$$\Lambda^\en\cong \Lambda\otimes_\F\Lambda\cong\Lambda^{**}\otimes_\F\Lambda\cong\Hom_\F(\Lambda^*,\Lambda)\fs$$
Finally, we use the tensor-hom adjunctions to obtain
$$\Ext^d_{\Lambda^\en}(\Lambda,\Lambda^\en)\cong\Ext^d_{\Lambda^\en}(\Lambda,\Hom_\F(\Lambda^*,\Lambda))\cong\Ext^d_{\Lambda}(\Lambda\otimes_\Lambda\Lambda^*,\Lambda)\cong\Ext^d_{\Lambda}(\Lambda^*,\Lambda)\fs$$
The second isomorphism is shown similarly.
\end{proof}

Using $E$, one can describe the functors $\tau_d$ and $\tau_d^-$ as follows.
\begin{proposition}\label{prop:adjunction}
If $\gldim\Lambda\leq d$ then we have isomorphisms of functors
$$\tau_d\cong\Hom_\Lambda(E,-):\Lambda\mMod\to\Lambda\mMod\ \mbox{ and }\ \tau_d^-\cong E\otimes_\Lambda-:\Lambda\mMod\to\Lambda\mMod.$$
In particular $\tau_d^-$ is left adjoint to $\tau_d$.
\end{proposition}
\begin{proof}
See the proof of \cite[Lemma 2.13]{io-stab}. The latter assertion follows from the former one.
\end{proof}

Now we recall the definition of higher preprojective algebras as given in \cite{io-napr}. 
\begin{definition}\label{ungraded preprojective algebra}
The \emph{higher preprojective algebra} (or, more precisely, the \emph{$(d+1)$-preprojective algebra}) of $\Lambda$ is the tensor algebra of the $\Lambda^\en$-module $E$:
$$\Pi=\Pi_{d+1}(\Lambda):=\Te_\Lambda(E).$$
Since this is a tensor algebra, it comes with a natural grading which we call the \emph{tensor grading}, i.e., the degree $i$ part of $\Pi$ is $E^i$.
\end{definition}

The following result justifies the name of the higher preprojective algebra.
\begin{proposition}\label{Pi as module2}
As both a left and a right $\Lambda$-module, $\Pi$ is the direct sum of all indecomposable $d$-preprojective $\Lambda$-modules.
\end{proposition}
\begin{proof}
The statement is immediate from the definition of $\Pi$ and Proposition \ref{prop:adjunction}.
\end{proof}

As in the global dimension $1$ case, the preprojective algebra can be identified with
$$\bigoplus_{i\geq0}\Hom_\Lambda(\Lambda,\tau_d^{-i}(\Lambda))$$
where the composition of $f:\Lambda\to\tau_d^{-i}(\Lambda)$ and $g:\Lambda\to\tau_d^{-j}(\Lambda)$ is given by
$$gf=\tau_d^{-i}(g)\circ f:\Lambda\to\tau_d^{-i-j}(\Lambda).$$ 
The $i$th part of the tensor grading is just $\Hom_\Lambda(\Lambda,\tau_d^{-i}(\Lambda))$.

\section{Description of higher preprojective algebras as higher Jacobi algebras}

The aim of this section is to introduce higher preprojective algebras and to give some of their basic properties, including presentations of these algebras by generators and relations.

\subsection{Preliminaries}\label{subsec:graded}

In this subsection, let $\Lambda$ be a finite dimensional $\F$-algebra $\Lambda$ with global dimension at most $d$, where $d$ is a positive integer.
Moreover we assume that
\[\Lambda=\tens_S(V)/I\]
for a separable $\F$-algebra $S$.  Thus $S^\en$ is semisimple, and the projective dimension of the $\Lambda^\en$-module $\Lambda$ coincides with the global dimension of $\Lambda$, which is at most $d$.
As before, let
\[E=\Ext^d_\Lambda(\Lambda^*,\Lambda).\]
We take a minimal projective resolution of the $\Lambda^{\en}$-module $\Lambda$:
\begin{equation}\label{minimal bimodule resolution}
0\to P_d\xrightarrow{\delta_d}\cdots\xrightarrow{\delta_3}P_2\xrightarrow{\delta_2}P_1\xrightarrow{\delta_1}P_0\to0\ \mbox{ with }\ P_i=\Lambda\otimes_SK_i\otimes_S\Lambda,
\end{equation}
where $K_i$ is an $S^\en$-module.
For each $i\ge1$, we define a map $\delta'_i$ by the commutative diagram
\begin{equation}\label{delta'}
\xymatrix@C=50pt{
\Hom_{\Lambda^\en}(P_{i-1},\Lambda^\en)\ar[r]^{_{\Lambda^\en}(\delta_i,\Lambda^\en)}\ar[d]^\sim & {}\Hom_{\Lambda^\en}(P_i,\Lambda^\en)\ar[d]^\sim\\
\Lambda\otimes_SK_{i-1}^\vee\otimes_S\Lambda\ar@{.>}[r]^{\delta_i'} & {}\Lambda\otimes_SK_i^\vee\otimes_S\Lambda,
}\end{equation}
where the vertical maps are given by Lemma \ref{lem-tensorhomiso}.

\begin{proposition}\label{prop-describee}
We have isomorphisms $E\cong(\Lambda\otimes_SK_d^\vee\otimes_S\Lambda)/\Image\delta_d'$ of $\Lambda^\en$-modules and $\hd E\cong K_d^\vee$ of $S^\en$-modules.
\end{proposition}

\begin{proof}
The former isomorphism is immediate from \eqref{delta'}.
Since the resolution \eqref{minimal bimodule resolution} is minimal,
$\Image\Hom_{\Lambda^\en}(\delta_d,\Lambda^\en)\subseteq J^\en(\Lambda\otimes_S K_d^\vee\otimes_S\Lambda)$ holds.
Thus $\hd E\cong\hd(\Lambda\otimes_S K_d^\vee\otimes_S\Lambda)=K_d^\vee$ since $S^\en$ is semisimple.
\end{proof}

Let $\overline{V}$ be the $S^\en$-module
\[\overline{V}:=V\oplus K_d^\vee.\]
This notation is meant to be reminiscent of $\overline{Q}$ which, in the global dimension $1$ case, is used to denote the doubled quiver of the underlying quiver $Q$.
For $T:=\Tens_S(V)$, we have an isomorphism 
\[\Tens_S(\overline{V})\cong\Tens_T(T\otimes_SK_d^\vee\otimes_ST)\]
by Lemma \ref{formula for Tens}(a). 
Regarding $T\otimes_SK_d^\vee\otimes_ST$ as a subspace of $\Tens_S(\overline{V})$, we have the following description of $\Pi$.

\begin{proposition}\label{prop:piquiv}
Let $\Lambda=T/I$ with $T=\Tens_S(V)$ and $I\subset V^{\ge2}$.
\begin{enumerate}[\rm(a)]
\item We have a surjective morphism of algebras:
\[\Tens_S(\overline{V})\onto\Pi\]
which is bijective on restriction to $S\oplus\overline{V}$.
\item Let $L$ be a subspace of $T\otimes_SK_d^\vee\otimes_ST$ whose image under the natural surjection $T\otimes_SK_d^\vee\otimes_ST\onto\mbox{$\Lambda\otimes_SK_d^\vee\otimes_S\Lambda$}$
is $\delta_d'(K_{d-1}^\vee)$. Then we have an isomorphism of algebras:
\[\Tens_S(\overline{V})/(I+L)\cong\Pi.\]
\end{enumerate}
\end{proposition}

\begin{proof}
We only need to prove part (b) of the proposition, from which part (a) will follow.
By Proposition \ref{prop-describee}, we have
\begin{eqnarray*}
E &\cong& \Cok\delta_d' = (\Lambda\otimes_SK_d^\vee\otimes_S\Lambda)/\Lambda\delta_d'(K_{d-1}^\vee)\Lambda\\
&\cong&(T\otimes_SK_d^\vee\otimes_ST)/(I\otimes_SK_d^\vee\otimes_ST+T\otimes_SK_d^\vee\otimes_SI+TLT).
\end{eqnarray*}
So, applying Lemma \ref{formula for Tens}(a) and (b), we have
\begin{eqnarray*}
&\Tens_S(\overline{V})/(I+L)\cong
\Tens_T(T\otimes_SK_d^\vee\otimes_ST)/(I+TLT)\\
&\cong\Tens_\Lambda((T\otimes_SK_d^\vee\otimes_ST)/(I\otimes_SK_d^\vee\otimes_ST+T\otimes_SK_d^\vee\otimes_SI+TLT))
\cong\Tens_\Lambda(E)=\Pi
\end{eqnarray*}
as desired.
\end{proof}

Consider the case where $\F$ is algebraically closed, so we can describe $\Lambda$ as $\F Q/I$.  Let $\{k_1,\ldots,k_r\}$ be a basis of $K_d$, each with a unique source and target $s(k_i)$ and $t(k_i)$, 
and let $\overline{Q}$ be the quiver obtained by adding $r$ arrows $k_i^*:t(k_i)\to s(k_i)$ to $Q$.
Then, just as $V$ is the arrow-space of $Q$, $\overline{V}$ is the arrow-space of $\overline{Q}$, and Proposition \ref{prop:piquiv} says that $\overline{Q}$ is the Gabriel quiver of $\Pi$.

We can therefore calculate the Gabriel quiver $\overline{Q}$ of $\Pi$ as follows.  First, for each vertex $i$ of $Q$, compute the projective resolution
$$0\to P_{i,n}\to\cdots\to P_{i,0}\to 0$$
of the simple left $\Lambda$-module $S_i$, where some projective modules $P_{i,h}$ may be zero.  Then, for each $i$ and for each summand of the projective module $P_{i,n}$ which is isomorphic to the projective cover of $S_j$, add an arrow $i\to j$ to the quiver $Q$.  The resulting quiver is $\overline{Q}$.

\begin{example}
Let
$$Q=\left[1\arr{\alpha}2\arr{\beta}3\arr{\gamma}4\arr{\delta}5\arr{\varepsilon}6\right]$$
and $\Lambda=\F Q/(\beta\gamma\delta,\gamma\delta\varepsilon)$.  Let $S_i$ denote the simple left $\Lambda$-module associated to the vertex $i$, and $P(S_i)$ its projective cover.  One can check that $\Lambda$ has global dimension $3$ and the only simple module with projective dimension $3$ is $S_6$.  Its projective resolution is
$$0\to P(S_2)\arr{\cdot \beta} P(S_3)\arr{\cdot {\gamma\delta}} P(S_5)\arr{\cdot \varepsilon} P(S_6)\to 0$$
where $\cdot a$ denotes right multiplication by $a$.
So the quiver $\overline{Q}$ of $\Pi$ is just $Q$ with an extra arrow from $6$ to $2$, which we label $(\beta\gamma\delta\varepsilon)^*$.
\end{example}

\subsection{Superpotentials and higher Jacobi algebras}\label{subsec:spots}

To introduce our main notions of superpotentials, we need preparations.
For an $\F$-algebra $A$ and an $A^\en$-module $M$, we write
\[c(M):=A\otimes_{A^\en}M\]
for the 0th Hochschild homology $H_0(A,M)$ of $A$.
This can be naturally identified with the quotient of $M$ modulo the subgroup generated by $am-ma$ with $a\in A$ and $m\in M$.
Therefore we have a natural surjective map $\pi:M\onto A\otimes_{A^\en}M$ of $\F$-modules.

For $A^\en$-modules $M_1,\ldots,M_\ell$, we clearly have functorial isomorphisms
\begin{equation}\label{cyclic}
c(M_1\otimes_A\cdots\otimes_AM_\ell)\cong c(M_2\otimes_A\cdots\otimes_AM_\ell\otimes_AM_1)\cong\cdots\cong c(M_\ell\otimes_AM_1\otimes_A\cdots\otimes_AM_{\ell-1}).
\end{equation}
For $M,N\in A^\en\mMod$, there is a functorial isomorphism
\begin{equation}\label{lem:tensae}
c(M\otimes_AN)\cong M\otimes_{A^\en}N
\end{equation}
given by $a\otimes(m\otimes n)\mapsto am\otimes n= m\otimes na$, whose inverse is $m\otimes n\mapsto 1\otimes(m\otimes n)$.
It gives a functorial morphism
\begin{equation}\label{c(MN)}
c(M\otimes_AN)\to\Hom_{A^\en}(M^\vee,N),
\end{equation}
which is an isomorphism if $M$ is a finitely generated projective $A^\en$-module.

Setting $M=A^\en$ in \eqref{lem:tensae}, we have a functorial isomorphism of $A^\en$-modules
\begin{equation}\label{c(Ae)}
c(A^\en\otimes_A N)\cong N.
\end{equation}
For $M,N\in A^\en\mMod$, we have a well-defined pairing
\begin{equation}\label{evaluation}
\ev_M\otimes 1_N:M^\vee\otimes_{\F}c(M\otimes_AN)\to N
\end{equation}
given as the composition $M^\vee\otimes_\F c(M\otimes_AN)\to c(A^\en\otimes_AN)\xrightarrow{\eqref{c(Ae)}} N$,
where the first map sends $f\otimes(1_A\otimes(m\otimes n))$ to $f(m) n$.

Now we are ready to introduce the following, which is a central notion in this paper.

\begin{definition}\label{define superpotential}
Let $S$ be a semisimple $\F$-algebra and $U$ an $S^\en$-module. A \emph{superpotential of degree $\ell$} for $T=\Tens_S(U)$ is an element of $c(U^\ell)=S\otimes_{S^\en}U^\ell$, where $U^\ell$ is the $\ell$th tensor power $U\otimes_S\cdots\otimes_S U$ as before.
\end{definition}

By \eqref{cyclic}, we have a well-defined automorphism
\[\rho:S\otimes_{S^\en}U^\ell\to S\otimes_{S^\en}U^\ell,\ (x_1\otimes x_2\otimes\cdots\otimes x_\ell) \mapsto (x_2\otimes\cdots\otimes x_\ell\otimes x_1).\]
Using $\rho$, we define $\varphi$ by
\[\varphi:=\sum_{i=0}^{\ell-1}(-1)^{(\ell-1)i}\rho^i:S\otimes_{S^\en}U^\ell\to S\otimes_{S^\en}U^\ell.\]
By \eqref{evaluation}, for $0\leq k\leq \ell$, we have a well-defined pairing
\[ \ev_{U^k}\otimes1_{U^{\ell-k}}:
(U^k)^\vee\otimes_\F c(U^\ell)\arr\sim
(U^k)^\vee\otimes_\F c(U^k\otimes_{S}U^{\ell-k})\xrightarrow{\ev_{U^k}\otimes1_{U^{\ell-k}}} U^{\ell-k}. \]
For $f\in(U^k)^\vee$ and $x\in c(U^\ell)$, we simply write $f\cdot x:=\ev_{U^k}\otimes1_{U^{\ell-k}}(f\otimes x)$. 

\begin{definition}\label{2}
Let $S$ be a semisimple $\F$-algebra, $U$ an $S^\en$-module, and $T=\Tens_S(U)$. For a superpotential $W$ of degree $\ell$ and a nonnegative integer $k\le\ell$, the \emph{$k$-Jacobi ideal} of $T$ is the two-sided ideal
$$J_S^k(U,W)=\left(f\cdot\varphi(W) \mid f\in\Hom_{S^\en}(U^k,S^\en)\right).$$
The \emph{$k$-Jacobi algebra} is the quotient algebra
$$P_S^k(U,W)=\Tens_S(U)/J_S^k(U,W).$$
\end{definition}

We now explain a connection to notation used elsewhere.

\begin{remark}
Given a quiver $Q$, we have a semisimple algebra $S=\F Q_0$ with basis the vertices of $Q$ and an $S^\en$-module $U=\F Q_1$ with basis the arrows.
For each $i\geq0$, let $Q_i$ be the set of all paths of length $i$ on $Q$.
Then $Q_i$ gives a basis of the $S^\en$-module $U^i$, and we denote by $\{p^\vee\mid p\in Q_i\}$ the dual basis of $(U^i)^\vee$ in the obvious sense.

Let $W$ be a superpotential for $\F Q=\Tens_S(U)$. Define
\[\partial_{p}W=p^\vee\cdot  \varphi(W).\]
Then the $k$-Jacobi ideal is the ideal of $\Tens_S(U)$ generated by $\{\partial_{p}W\st p\in Q_k\}$.
Note that when $k=1$ and the superpotential $W$ is of odd degree we recover the usual notion of the Jacobi algebra of a quiver with potential $(Q,W)$.

Note also that some sources, such as \cite{bsw}, define the superpotential to be $\varphi(W)$ rather than $W$.
\end{remark}

In the rest, we give general observations which will be used later. Let $A$ be an $\F$-algebra.

\begin{lemma}\label{XYZ}
For $A^\en$-modules $X,Y,Z$, we have functorial morphisms
\[\Hom_{A^\en}(X^\vee,Y\otimes_AZ)\leftarrow c(X\otimes_AY\otimes_AZ)\to\Hom_{A^\en}(Y^\vee,Z\otimes_AX).\]
The left (respectively, right) one is an isomorphism if $X$ (respectively, $Y$) is a projective $A^\en$-module. 
\end{lemma}
\begin{proof}
Using \eqref{c(MN)}, we have functorial morphisms $c(X\otimes_AY\otimes_AZ)\to\Hom_{A^\en}(X^\vee,Y\otimes_AZ)$ and $c(X\otimes_AY\otimes_AZ)\stackrel{\eqref{cyclic}}{\cong}c(Y\otimes_AZ\otimes_AX)\to\Hom_{A^\en}(Y^\vee,Z\otimes_AX)$.
\end{proof}

As in \eqref{XY4} when $A$ is semisimple, for $A^\en$-modules $X,Y,Z$, we have functorial isomorphisms
\begin{equation}\label{XY}
Y^\vee\otimes_A\Hom_A(X,A)\cong(X\otimes_AY)^\vee\ \mbox{ and }\ \Hom_{A^{\op}}(Y,A)\otimes_AX^\vee\cong(X\otimes_AY)^\vee.
\end{equation}
The first map sends
$f\otimes g$ to $(x\otimes y\mapsto\sum_ig(xs_i)\otimes s'_i)$ where $f(y)=\sum_is_i\otimes s'_i$,
and the second one sends $f'\otimes g'$ to $(x\otimes y\mapsto\sum_jt_j\otimes f'(t'_jy))$ where $g'(x)=\sum_jt_j\otimes t'_j$.
We have the following commutative diagram.
\begin{equation}\label{XYZ commute}
\xymatrix@C4em{(\Hom_{A^{\op}}(Y,A)\otimes_AX^\vee)\otimes_\F c(X\otimes_AY\otimes_AZ)\ar[r]^(.58){1\otimes\ev_X\otimes1}\ar[d]^{\eqref{XY}}_\wr&\Hom_{A^{\op}}(Y,A)\otimes_A (Y\otimes_AZ)\ar[d]^{\ev_Y\otimes1}\\
(X\otimes_AY)^\vee\otimes_\F c(X\otimes_AY\otimes_AZ)\ar[r]^(.6){\ev_{X\otimes Y}}&Z\\
(Y^\vee\otimes_A\Hom_A(X,A))\otimes_\F c(X\otimes_AY\otimes_AZ)\ar[r]^(.58){\ev_Y\otimes1\otimes1}\ar[u]_{\eqref{XY}}^\wr&(Z\otimes_AX)\otimes_A\Hom_A(X,A)\ar[u]_{1\otimes\ev_X}}
\end{equation}

\subsection{Higher preprojective algebras of Koszul algebras}\label{subsec:hpka}
Now we go back to the setting in Section \ref{subsec:graded}, that is, $\Lambda$ is a finite dimensional $\F$-algebra with global dimension $d>0$.
Moreover we assume that $\Lambda$ is a Koszul algebra and
\[\Lambda=\tens_S(V)/I\]
for a separable $\F$-algebra $S$.
Then the minimal $\Z$-graded projective resolution \eqref{minimal bimodule resolution} of the $\Lambda^\en$-module $\Lambda$ is given by the Koszul bimodule complex \eqref{Koszul bimodule complex}.
Let $\overline{V}=V\oplus K_d^\vee$.

\begin{definition}\label{assoc-superpot}
We define a superpotential $W$ of degree $d+1$ for $\tens_S(\overline{V})$ as the image of $1_\F\in\F$ under the composition
\[\F\xrightarrow{\coev_{K_d}} c(K_d\otimes_SK_d^\vee)
\subset c(V^d\otimes_SK_d^\vee)\subset
c(\overline{V}^{d}\otimes_S\overline{V})=c(\overline{V}^{d+1}),\]
where $\coev_{K_d}:\F\to\End_{S^\en}(K_d)\cong K_d\otimes_{S^\en}K_d^\vee\cong c(K_d\otimes_SK_d^\vee)$ is the coevaluation map. We call $W$ the \emph{superpotential associated to $\Lambda$}, or the \emph{associated superpotential}.  
\end{definition}

By Lemma \ref{XYZ}, we have isomorphisms
$\Hom_{S^\en}(K_i,V\otimes_SK_{i-1})\cong\Hom_{S^\en}(K_{i-1}^\vee,K_i^\vee\otimes_SV)$ and
$\Hom_{S^\en}(K_i,K_{i-1}\otimes_SV)\cong\Hom_{S^\en}(K_{i-1}^\vee,V\otimes_SK_i^\vee)$.
Thus the inclusions
 $\iota_i^\ell:K_i\to V\otimes_SK_{i-1}$ and $\iota_i^r:K_i\to K_{i-1}\otimes_SV$ give rise to
\begin{equation}\label{theta}
\theta_i^\ell:K_{i-1}^\vee\to K_i^\vee\otimes_SV\ \mbox{ and }\ \theta_i^r:K_{i-1}^\vee\to V\otimes_SK_i^\vee.\end{equation}
We will need the following observations.

\begin{lemma}\label{derivation}
The following assertions hold.
\begin{enumerate}[\rm(a)]
\item The map $(\overline{V}^{d-1})^\vee\onto K_{d-1}^\vee\xrightarrow{\theta_d^r}V\otimes_SK_d^\vee\into\overline{V}^2$ coincides with $-\cdot W:(\overline{V}^{d-1})^\vee\to\overline{V}^2$.
\item The map $(\overline{V}^{d-1})^\vee\onto K_{d-1}^\vee\xrightarrow{\theta_d^\ell}K_d^\vee\otimes_SV\into\overline{V}^2$ coincides with $-\cdot\rho(W):(\overline{V}^{d-1})^\vee\to\overline{V}^2$.
\end{enumerate}
\end{lemma}

\begin{proof}
(a) By definition, $W$ belongs to the subspace $c(K_{d-1}\otimes_SV\otimes_SK_d^\vee)$ of $c(\overline{V}^{d+1})$, and coincides with $\iota_d^r$ under the isomorphism $\Hom_{S^\en}(K_d,K_{d-1}\otimes_SV)\cong c(K_{d-1}\otimes_SV\otimes_SK_d^\vee)$ in Lemma \ref{XYZ}.
By definition, $\theta_d^r$ is the image of $W$ under the isomorphism $c(K_{d-1}\otimes_SV\otimes_SK_d^\vee)\cong\Hom_{S^\en}(K_{d-1}^\vee,V\otimes_SK_d^\vee)$ in Lemma \ref{XYZ}. Thus $\theta_d^r$ coincides with
\[K_{d-1}^\vee\xrightarrow{1\otimes W}K_{d-1}^\vee\otimes_\F c(K_{d-1}\otimes_SV\otimes_SK_d^\vee)
\xrightarrow{\ev_{K_{d-1}}\otimes1\otimes1}V\otimes_SK_d^\vee.\]
On the other hand, since $W$ belongs to $c(K_{d-1}\otimes_SV\otimes_SK_d^\vee)$, the map $-\cdot W$ factors through $K_{d-1}^\vee$. Thus the assertion follows.

(b) Although the argument is mostly the same as (a), we record the details.

By definition, $W$ belongs to the subspace $c(V\otimes_SK_{d-1}\otimes_SK_d^\vee)$ of $c(\overline{V}^{d+1})$, and coincides with $\iota_d^\ell$ under the isomorphism $\Hom_{S^\en}(K_d,V\otimes_SK_{d-1})\cong c(V\otimes_SK_{d-1}\otimes_SK_d^\vee)$ in Lemma \ref{XYZ}.
By definition, $\theta_d^\ell$ is the image of $W$ under the isomorphism $c(V\otimes_SK_{d-1}\otimes_SK_d^\vee)\cong\Hom_{S^\en}(K_{d-1}^\vee,K_d^\vee\otimes_SV)$ in Lemma \ref{XYZ}. Thus $\theta_d^\ell$ coincides with
\[K_{d-1}^\vee\xrightarrow{1\otimes W}K_{d-1}^\vee\otimes_\F c(V\otimes_SK_{d-1}\otimes_SK_d^\vee)\xrightarrow{1\otimes\rho}K_{d-1}^\vee\otimes_\F c(K_{d-1}\otimes_SK_d^\vee\otimes_SV)\xrightarrow{\ev_{K_{d-1}}\otimes1\otimes1}V\otimes_SK_d^\vee,\]
which equals $K_{d-1}^\vee\xrightarrow{1\otimes\rho(W)}K_{d-1}^\vee\otimes_\F c(K_{d-1}\otimes_SK_d^\vee\otimes_SV)\xrightarrow{\ev_{K_{d-1}}\otimes1\otimes1}V\otimes_SK_d^\vee$.

On the other hand, since $\rho(W)$ belongs to $c(K_{d-1}\otimes_SK_d^\vee\otimes_SV)$, the map $-\cdot\rho(W)$ factors through $K_{d-1}^\vee$. Thus the assertion follows.
\end{proof}

On the other hand, $\theta_i^\ell$ and $\theta_i^r$ induce morphisms
$\hat{\theta}^\ell_i$ and $\hat{\theta}^r_i:\Lambda\otimes_SK_{i-1}^\vee\otimes_S\Lambda\to\Lambda\otimes_SK_{i}^\vee\otimes_S\Lambda$
of $\Lambda^\en$-modules. 
This gives an explicit construction of $\delta_i'$ in \eqref{delta'} by the following observation. 
\begin{lemma}\label{lem:easydual}
If $\Lambda$ is Koszul then we have a commutative diagram
$$\xymatrix @C=60pt{
\Hom_{\Lambda^\en}(\Lambda\otimes_SK_{i-1}\otimes_S\Lambda,\Lambda^\en)\ar[r]^{_{\Lambda^\en}(\delta_i,\Lambda^\en)}\ar[d]^\sim & {}\Hom_{\Lambda^\en}(\Lambda\otimes_SK_{i}\otimes_S\Lambda,\Lambda^\en)\ar[d]^\sim\\
\Lambda\otimes_SK_{i-1}^\vee\otimes_S\Lambda\ar[r]^{\hat{\theta}^\ell_i+(-1)^i\hat{\theta}^r_i} & {}\Lambda\otimes_SK_{i}^\vee\otimes_S\Lambda 
}$$
and therefore $\delta_i'=\hat{\theta}^\ell_i+(-1)^i\hat{\theta}^r_i$.
\end{lemma}

To prove this, we prepare the following observation.

\begin{lemma}\label{XYV}
For $S^\en$-modules $X$ and $Y$, we have the following commutative diagram.
\[\begin{xy}
(0,0)*+{c(Y\otimes_SV\otimes_SX^\vee)}="0",
(0,-12)*+{\Hom_{S^\en}(X,Y\otimes_SV)}="1",
(80,0)*+{\Hom_{S^\en}(Y^\vee,V\otimes_SX^\vee)}="2",
(0,-24)*+{\Hom_{\Lambda^\en}(\Lambda\otimes_SX\otimes_S\Lambda,\Lambda\otimes_SY\otimes_S\Lambda)}="3",
(80,-12)*+{\Hom_{\Lambda^\en}(\Lambda\otimes_SY^\vee\otimes_S\Lambda,\Lambda\otimes_SX^\vee\otimes_S\Lambda)}="4",
(80,-24)*+{\Hom_{\Lambda^\en}((\Lambda\otimes_SY\otimes_S\Lambda)^{\vee_\Lambda},(\Lambda\otimes_SX\otimes_S\Lambda)^{\vee_\Lambda}),}="5",
\ar"0";"1",^{{\rm Lem.\ \ref{XYZ}}}
\ar"0";"2",^{{\rm Lem.\ \ref{XYZ}}}
\ar"1";"3",^{\hat{(-)}}
\ar"2";"4",^{\hat{(-)}}
\ar"3";"5",^(.45){(-)^{\vee_\Lambda}}_(.45)\sim
\ar"4";"5",^{{\rm Lem.\ \ref{lem-tensorhomiso}}}_\sim
\end{xy}\]
where we write $(-)^{\vee_\Lambda}=\Hom_{\Lambda^\en}(-,\Lambda^\en)$.
\end{lemma}

\begin{proof}
Fix $y\otimes v\otimes f\in c(Y\otimes_SV\otimes_SX^\vee)$, and let $a\in\Hom_{\Lambda^\en}(\Lambda\otimes_SX\otimes_S\Lambda,\Lambda\otimes_SY\otimes_S\Lambda)$ and $b\in\Hom_{\Lambda^\en}(\Lambda\otimes_SY^\vee\otimes_S\Lambda,\Lambda\otimes_SX^\vee\otimes_S\Lambda)$
be the corresponding maps.
Let $a'$ and $b'$ be the maps in $\Hom_{\Lambda^\en}((\Lambda\otimes_SY\otimes_S\Lambda)^{\vee_\Lambda},(\Lambda\otimes_SX\otimes_S\Lambda)^{\vee_\Lambda})$ correponding to $a$ and $b$ respectively. 
To prove $a'=b'$, it suffices to show that $a'(1\otimes g\otimes1)=b'(1\otimes g\otimes1)$ holds for all $g\in Y^\vee$, where $1\otimes g\otimes 1\in(\Lambda\otimes_SY\otimes_S\Lambda)^{\vee_\Lambda}$ is the natural extension of $g$.

Since $a(1\otimes x\otimes 1)=(1\otimes y\otimes v)f(x)=(1\otimes y\otimes1)((v\otimes1)f(x))$ holds for all $x\in X$, we have $(a'(1\otimes g\otimes 1))(1\otimes x\otimes1)=g(y)(v\otimes1)f(x)$.
On the other hand, since $b(1\otimes g\otimes1)=g(y)(v\otimes f\otimes1)=(g(y)(v\otimes1))(1\otimes f\otimes1)$ holds for all $g\in Y^\vee$, we have $(b'(1\otimes g\otimes 1))(1\otimes x\otimes1)=g(y)(v\otimes1)f(x)$. Thus $a'=b'$ holds.
\end{proof}

\begin{proof}[Proof of Lemma \ref{lem:easydual}]
Since $\delta_i=\hat{\iota}_i^\ell+(-1)^i\hat{\iota}_i^r$, it suffices to show that the following diagram commutes for $s\in\{\ell,r\}$.
$$\xymatrix @C=60pt{
\Hom_{\Lambda^\en}(\Lambda\otimes_SK_{i-1}\otimes_S\Lambda,\Lambda^\en)\ar[r]^{_{\Lambda^\en}(\hat{\iota}_i^s,\Lambda^\en)}\ar[d]^\sim & {}\Hom_{\Lambda^\en}(\Lambda\otimes_SK_{i}\otimes_S\Lambda,\Lambda^\en)\ar[d]^\sim\\
\Lambda\otimes_SK_{i-1}^\vee\otimes_S\Lambda\ar[r]^{\hat{\theta}^s_i} & {}\Lambda\otimes_SK_{i}^\vee\otimes_S\Lambda.
}$$
We just show the $s=r$ version; $s=\ell$ is the dual.
We apply Lemma \ref{XYV} to $X:=K_i$ and $Y:=K_{i-1}$.
Since $\iota_i^r\in\Hom_{S^\en}(K_i,K_{i-1}\otimes_SV)$ corresponds to $\theta_i^r\in\Hom_{S^\en}(K_{i-1}^\vee,V\otimes_SK_i^\vee)$,
the map ${}_{\Lambda^\en}(\hat{\iota}_i^r,\Lambda^\en)$ coincides with $\hat{\theta}_i^r$ up to the isomorphisms in Lemma \ref{lem-tensorhomiso}.
This gives the commutativity of the above diagram.
\end{proof}

Since $\Lambda$ is Koszul, we can regard
$\delta_d'(K_{d-1}^\vee)\subset V\otimes_SK_d^\vee\otimes_SS+S\otimes_SK_d^\vee\otimes_SV$
as a subspace of $T\otimes_SK_d^\vee\otimes_ST\subset\Tens_S(\overline{V})$ naturally.
Now we show the following assertion.

\begin{proposition}\label{prop:pi-quadratic}
If $\Lambda=\Tens_S(V)/(R)$ is a finite-dimensional Koszul algebra  of global dimension $d$ with $R\subset V^2$, then we have an isomorphism of algebras:
\[\Pi\cong\Tens_S(\overline{V})/(R+\delta_d'(K_{d-1}^\vee)).\]
In particular, $\Pi$ is quadratic.
\end{proposition}

\begin{proof}
The first assertion is immediate from Proposition \ref{prop:piquiv}(b). The second assertion is immediate from the first one since both $R$ and $\delta_d'(K_{d-1}^\vee)$ are contained in $\overline{V}^2$.
\end{proof}

Now we are ready to prove the following.

\begin{theorem}\label{thm:pi-as-alg-pot}
If $\Lambda=\Tens_S(V)/(R)$ is a finite-dimensional Koszul algebra of global dimension $d$, then we have an isomorphism of algebras:
$$\Pi\cong P_S^{d-1}(\overline{V},W)/(R).$$ 
\end{theorem}

\begin{proof}
The left-hand side is $\Tens_S(\overline{V})/(R+\delta_d'(K_{d-1}^\vee))$ by Proposition \ref{prop:pi-quadratic}, and the right-hand side is
$\Tens_S(\overline{V})/(R+(\overline{V}^{d-1})^\vee\cdot\varphi(W))$ by definition.
It suffices to prove
$R+\delta_d'(K_{d-1}^\vee)=R+(\overline{V}^{d-1})^\vee\cdot\varphi(W)$.

As $K_d=\bigcap_{i=2}^{d}V^{i-2}\otimes_SR\otimes_SV^{d-i}$,
for each $2\le i\le d$ we have 
$$W\in c(K_d\otimes_SK_d^\vee)\subset c(V^{i-2}\otimes_SR\otimes_SV^{d-i}\otimes_SK_d^\vee)$$ 
and hence $\rho^i(W)\in c(V^{d-i}\otimes_SK_d^\vee\otimes_SV^{i-2}\otimes_SR)$.
Therefore $(\overline{V}^{d-1})^\vee\cdot\rho^i(W)\subset R$ holds. In particular,
\begin{eqnarray*}
R+(\overline{V}^{d-1})^\vee\cdot\varphi(W)&=&R+(\overline{V}^{d-1})^\vee\cdot(W+(-1)^d\rho(W))\stackrel{{\rm Lem.\ \ref{derivation}}}{=}R+(\theta_d^r+(-1)^d\theta_d^\ell)(K_{d-1}^\vee)\\
&\stackrel{{\rm Lem.\ \ref{lem:easydual}}}{=}&R+\delta_d'(K_{d-1}^\vee)
\end{eqnarray*}
holds as desired.
\end{proof}

The extension condition in the following theorem is a special case of the following property of \cite[Section 3]{io-stab}. 
Given a $d$-cluster tilting subcategory $\curlU$ of $\Db(\Lambda)$, we say that $\curlU$ has the \emph{vosnex property} (``vanishing of small negative extensions'') if $\Hom_{\Db(\Lambda)}(\curlU[j],\curlU)=0$ for $j\in\{1,2,\ldots, d-2\}$. 
In this case, since $\Lambda,\Lambda^*[-d]\in\curlU$, we have $\Ext^{d-j}_{\Lambda^\en}(\Lambda,\Lambda^\en)\cong\Ext^{d-j}_{\Lambda}(\Lambda^*,\Lambda)\cong\Hom_{\Db(\Lambda)}(\Lambda^*[j-d],\Lambda)=0$ for $j\in\{1,2,\ldots, d-2\}$.

\begin{theorem}\label{thm:qp-vosnex}
Suppose $\Lambda$ is a finite-dimensional Koszul algebra of global dimension $d$. 
If $\Ext^{i}_{\Lambda^\en}(\Lambda,\Lambda^\en)_{-i}=0$ for $2\leq i\leq d-1$, then we have an isomorphism of  algebras:
\[\Pi\cong P_S^{d-1}(\overline{V},W).\]
\end{theorem}

\begin{proof}
By Theorem \ref{thm:pi-as-alg-pot}, it suffices to prove $(\overline{V}^{d-1})^\vee\cdot\varphi(W)\supseteq R$.
In fact, for each $2\le i\le d$, we prove by downwards induction
\begin{equation}\label{Z(i)}
(\overline{V}^{d-i+1})^\vee\cdot\varphi(W)\supseteq K_i.
\end{equation}

First we prove \eqref{Z(i)} for $i=d$.
Consider the decomposition $\overline{V}^\vee=V^\vee\oplus K_d$.
Since $W=\coev_{K_d}(1_\F)$, we have $K_d\cdot\rho^d(W)=K_d$ and
$K_d\cdot\rho^i(W)=0$ for each $0\le i\le d-1$.
Thus $\overline{V}^\vee\cdot\varphi(W)\supseteq K_d\cdot\varphi(W)=K_d$ holds.

Next, for each $3\le i\le d$, we prove
\begin{equation}\label{Ki to Ki-1}
\Hom_{S^\op}(V,S)\cdot K_{i}+K_{i}\cdot\Hom_S(V,S)=K_{i-1},
\end{equation}
where $\cdot$ are the maps $\Hom_{S^\op}(V,S)\otimes_SV^i\to V^{i-1}$ and $V^i\otimes_S\Hom_{S}(V,S)\to V^{i-1}$ given by the evaluations. 
We use the Koszul resolution together with Lemma \ref{lem:easydual}.
These tell us that $\Ext^{i-1}_{\Lambda^\en}(\Lambda,\Lambda^\en)$ is the $(i-1)$st homology of the complex
\[
0 \leftarrow \Lambda\otimes_SK_{d}^\vee\otimes_S\Lambda  \leftarrow \Lambda\otimes_SK_{d-1}^\vee\otimes_S\Lambda  \leftarrow \cdots \leftarrow \Lambda\otimes_SK_{1}^\vee\otimes_S\Lambda  \leftarrow \Lambda\otimes_SK_{0}^\vee \otimes_S\Lambda   \leftarrow 0
\]
where the differentials are induced by the maps
\[\delta'_i=\theta_i^\ell+(-1)^i\theta_i^r:K_{i-1}^\vee\to(K_{i}^\vee\otimes_SV)\oplus(V\otimes_SK_{i}^\vee)\subset\Lambda\otimes_SK_{i}^\vee\otimes_S\Lambda.\]
This is injective since its kernel is $\Ext^{i-1}_{\Lambda^\en}(\Lambda,\Lambda^\en)_{1-i}=0$ by our assumption.
Applying $(-)^\vee$, we have a surjective map
\[(\Hom_{S^\op}(V,S)\otimes_SK_{i})\oplus(K_{i}\otimes_S\Hom_S(V,S))\stackrel{\eqref{XY}}{\cong}(K_{i}^\vee\otimes_SV)^\vee\oplus(V\otimes_SK_{i}^\vee)^\vee\xrightarrow{(\delta'_i)^\vee} K_{i-1}.\]
This is a restriction of the map $(\Hom_{S^\op}(V,S)\otimes_SV^i)\oplus(V^i\otimes_S\Hom_{S}(V,S))\to V^{i-1}$ given by the evaluations. Thus \eqref{Ki to Ki-1} holds.

Now assume \eqref{Z(i)} holds.
Applying the upper part of \eqref{XYZ commute} to $(X,Y,Z)=(\overline{V}^{d-i+1},\overline{V},\overline{V}^{i-1})$ and the lower one to $(X,Y,Z)=(\overline{V},\overline{V}^{d-i+1},\overline{V}^{i-1})$ respectively, we obtain
\begin{eqnarray*}
&(\overline{V}^{d-i+2})^\vee\cdot\varphi(W)\stackrel{\eqref{XYZ commute}}{=}
\Hom_{S^\op}(\overline{V},S)\cdot((\overline{V}^{d-i+1})^\vee\cdot\varphi(W))\
\stackrel{\eqref{Z(i)}}{\supseteq} \Hom_{S^\op}(\overline{V},S)\cdot K_i,&\\
&(\overline{V}^{d-i+2})^\vee\cdot\varphi(W)\stackrel{\eqref{XYZ commute}}{=}
((\overline{V}^{d-i+1})^\vee\cdot\varphi(W))\cdot\Hom_S(\overline{V},S)\stackrel{\eqref{Z(i)}}{\supseteq} K_i\cdot\Hom_S(\overline{V},S).&
\end{eqnarray*}
Thus $(\overline{V}^{d-i+2})^\vee\cdot\varphi(W)\supseteq \Hom_{S^\op}(V,S)\cdot K_i+K_i\cdot\Hom_S(V,S)\stackrel{\eqref{Ki to Ki-1}}{=}K_{i-1}$ holds, which completes the induction.
\end{proof}

Note that the condition of Theorem \ref{thm:qp-vosnex} is vacuous when $d=2$, so this result agrees with Keller's description of $3$-preprojective algebras (see \cite[Theorem 6.10]{kel-dcy} and \cite[Section 2.2]{hi-selfinj}).

We will see in Corollary \ref{cor:kos-nher-prep} that this theorem is particularly applicable to $d$-hereditary algebras.

\begin{example}\label{example of associated superpotential}
(a) Consider the quiver
$$Q=[\xymatrix@C=1.5em{1\ar[r]^{\alpha}&2\ar[r]^{\beta}&3\ar[r]^{\gamma}&4}]$$
and the algebra $\Lambda=\F Q/(\alpha\beta,\beta\gamma)$.  One can check that
it satisfies the conditions of Theorem \ref{thm:qp-vosnex} for $d=3$.
(In fact $\Lambda$ is Koszul and $3$-representation finite, see Definition \ref{define d-RF and d-RI} below.)
 We have $K_3=\gen{\alpha\beta\gamma}$ so the quiver $\overline{Q}$ of $\Pi=\Pi(\Lambda)$ is 
$$\overline{Q}=[\xymatrix@C=1.5em{1\ar[r]^{\alpha}&2\ar[r]^{\beta}&3\ar[r]^{\gamma}&4\ar@/^1pc/[lll]^{\eta}}]$$
where $\eta=(\alpha\beta\gamma)^*$. The superpotential $W$ is represented by  $\alpha\beta\gamma{\eta}$ and the space of relations of $\Pi$ is given by $\overline{V}^{-2}\cdot \varphi(W)  =\gen{\alpha\beta,\beta\gamma,\gamma{\eta},{\eta}\alpha}$.

(b) Next consider the quiver
$$Q=[\xymatrix@C=1.5em{1\ar[r]^{\alpha}&2\ar[r]^{\beta}&3\ar[r]^{\gamma}&4\ar[r]^{\delta}&5\ar[r]^{\varepsilon}&6}]$$
and the algebra $\Lambda=\F Q/(R)$ with $R=\langle\alpha\beta,\beta\gamma,\delta\varepsilon\rangle$.  One can check that $\Lambda$ has global dimension $3$ and is Koszul, but it does not satisfy 
the condition of Theorem \ref{thm:qp-vosnex} as $\Ext^2_{\Lambda}(\Lambda^* e_6,\Lambda e_4)_{-2}\neq0$.
Again, we have $K_3=\gen{\alpha\beta\gamma}$ so the quiver $\overline{Q}$ of $\Pi(\Lambda)$ is 
$$\overline{Q}=[\xymatrix@C=1.5em{1\ar[r]^{\alpha}&2\ar[r]^{\beta}&3\ar[r]^{\gamma}&4\ar@/^1pc/[lll]^{\eta}\ar[r]^{\delta}&5\ar[r]^{\varepsilon}&6}]$$
where ${\eta=}(\alpha\beta\gamma)^*$.  The superpotential $W$ is represented by  $\alpha\beta\gamma{\eta}$ and the $2$-Jacobi ideal is generated by $\overline{V}^{-2}\cdot \varphi(W) =\gen{\alpha\beta,\beta\gamma,\gamma{\eta},{\eta}\alpha}$.  We see that this doesn't include $\delta\varepsilon$, and so to obtain the whole space of relations of $\Pi$ we need to consider $R+\overline{V}^{-2}\cdot  \varphi(W)$.
\end{example}

\begin{remark}
It is worth pointing out that higher preprojective algebras are sometimes higher Jacobi algebras even in the non-Koszul case.  For example, consider the following example, due to Vaso \cite[Example 5.3]{vas}, of an algebra of global dimension 4 which satisfies the Ext-vanishing condition of Theorem \ref{thm:qp-vosnex}. (In fact $\Lambda$ is $4$-representation finite, see Definition \ref{define d-RF and d-RI} below.)
We take the quiver
 $$Q=[\xymatrix@C=1.5em{1\ar[r]^{\alpha}&2\ar[r]^{\beta}&3\ar[r]^{\gamma}&4\ar[r]^{\delta}&5\ar[r]^{\varepsilon}&6\ar[r]^{\zeta}&7\ar[r]^{\eta}&8\ar[r]^{\theta}&9}]$$
and the algebra $\Lambda=\F Q/(\rad \F Q)^4=\F Q/(\alpha\beta\gamma\delta, \beta\gamma\delta\varepsilon, \gamma\delta\varepsilon\zeta, \delta\varepsilon\zeta\eta, \varepsilon\zeta\eta\theta)$.
We know from Proposition \ref{prop:piquiv} that the quiver for $\Pi$ is
 $$\overline{Q}=[\xymatrix@C=1.5em{1\ar[r]^{\alpha}&2\ar[r]^{\beta}&3\ar[r]^{\gamma}&4\ar[r]^{\delta}&5\ar[r]^{\varepsilon}&6\ar[r]^{\zeta}&7\ar[r]^{\eta}&8\ar[r]^{\theta}&9\ar@/^1pc/[llllllll]^{\iota}}]$$
where $\iota=(\alpha\beta\gamma\delta\varepsilon\zeta\eta\theta)^*$, and one can check that $\Pi$ is in fact a $5$-Jacobi algebra: we obtain its relations by differentiating the superpotential represented by $W=\alpha\beta\gamma\delta\varepsilon\zeta\eta\theta\iota$ with respect to paths of length $5$.

We do not know an example of a non-Koszul algebra which satisfies the ext-vanishing condition of Theorem \ref{thm:qp-vosnex} but is not a higher Jacobi algebra. 
\end{remark}

\section{Resolutions of simple modules over higher preprojective algebras}
The aim of this section is to construct projective resolutions of simple modules for preprojective algebras of $d$-hereditary algebras.

\subsection{Preliminaries on $d$-hereditary algebras }
Let $\Lambda$ be a finite dimensional $\F$-algebra with $\gl\Lambda\le d$, and $\Db(\Lambda)$ the derived category of finitely generated left $\Lambda$-modules with bounded homology. Then we have the following result on formality.

\begin{lemma}{\cite[Lemma 5.2]{iya-ct-higher}}\label{formality}
If $X\in\der(\Lambda)$ satisfies $H^i(X)=0$ for any $i\notin d\Z$, then $X\cong \bigoplus_{i\in d\Z}H^i(X)[-i]$.
\end{lemma}

Let $\nu$ denote the Nakayama functor
$$\nu:=\Lambda^*\dert_\Lambda-:\der(\Lambda)\arr{\sim}\der(\Lambda)$$
of $\Lambda$ and let
$\nu^{-1}$ denote its quasi-inverse, defined using the internal hom,
$$\nu^{-1}:=\RHom_\Lambda(\Lambda^*,-):\der(\Lambda)\arr{\sim}\der(\Lambda).$$
Let $\nu_d$ denote the shifted Nakayama functor $\nu_d=\nu\circ[-d]$ and $\nu_d^{-1}=\nu^{-1}\circ[d]$. Then we have
\begin{equation}\label{tau and nu}
\tau_d=H^0(\nu_d-):\mod\Lambda\to\mod\Lambda\ \mbox{ and }\ \tau_d^-=H^0(\nu_d^{-1}-):\mod\Lambda\to\mod\Lambda.
\end{equation}

\begin{definition}{\cite[Definition 3.2]{hio}}\label{def:nhered}
A finite dimensional algebra $\Lambda$ with $\gldim\Lambda= d$ is \emph{$d$-hereditary} if $H^i(\nu_d^j(\Lambda))=0$ for all $i,j\in \Z$ such that $i\notin d\Z$.
\end{definition}

One of the important properties of $d$-hereditary algebras $\Lambda$ follows from Lemma \ref{formality}: for any $j\in\Z$ and an indecomposable projective $\Lambda$-module $P$, there exists $i\in\Z$ such that
\begin{equation}\label{eq:nu_d^j(P)}
\nu_d^j(P)\cong H^{di}(\nu_d^j(P))[-di].
\end{equation}
Note that in \cite{hio}, the weaker condition $\gldim\Lambda\leq d$ instead of $\gldim\Lambda=d$ was imposed. The only difference between the two definitions is whether we allow $\Lambda$ to be semisimple, which is a case we are not interested in. Therefore we always assume $\gldim\Lambda=d$.

The following result is an immediate consequence of Theorem \ref{thm:qp-vosnex}.

\begin{corollary}\label{cor:kos-nher-prep}
Let $\Lambda=\Tens_S(V)/(R)$ be a Koszul $d$-hereditary algebra over a separable $\F$-algebra $S$ and $(\overline{V},W)$ the associated superpotential. Then we have $\Pi\cong P_S^d(\overline{V},W)$.
\end{corollary}

\begin{proof}
The assertion is immediate from Theorem \ref{thm:qp-vosnex} since
$$\Ext^{d-i}_{\Lambda^\en}(\Lambda,\Lambda^\en)\cong\Hom_{\der(\Lambda)}(\Lambda^*,\Lambda[d-i])\cong\Hom_{\der(\Lambda)}(\Lambda[i],\nu_d^{-1}(\Lambda))\cong
H^i(\nu_d^{-1}(\Lambda))=0$$
holds for any $0<i<d$.
\end{proof}

\begin{definition}{\cite{io-stab,hio}}\label{define d-RF and d-RI}
We say that a finite-dimensional $\F$-algebra $\Lambda$ with $\gldim\Lambda=d$ is:
\begin{enumerate}[$\bullet$]
 \item \emph{$d$-representation finite} (or \emph{$d$-RF}) if there exists an \emph{$d$-cluster tilting} $\Lambda$-module $M$, that is, 
\begin{eqnarray*}
\add M&=&\left\{X\in\Lambda\mMod \st \Ext_\Lambda^i(X,M)=0\text{ for all }0<i<d\right\} \\
      &=&\left\{Y\in\Lambda\mMod \st \Ext_\Lambda^i(M,Y)=0\text{ for all }0<i<d\right\}\fs
\end{eqnarray*}
 \item \emph{$d$-representation infinite} (or \emph{$d$-RI}) if $\nu_d^{-i}(\Lambda)$ is concentrated in degree $0$ for any $i\geq0$.
\end{enumerate}
\end{definition}

Then we have a dichotomy theorem:
\begin{theorem}{\cite[Theorem 3.4]{hio}}
Every ring-indecomposable finite-dimensional $\F$-algebra is $d$-hereditary if and only if it is either $d$-RF or $d$-RI.
\end{theorem}

In the study of $d$-hereditary algebras, the subcategory
$$\UU:=\add\{\nu_d^i(\Lambda)\mid i\in\Z\}$$
of $\der(\Lambda)$ plays an important role.

We give a few properties of $\UU$ and the categories $\PP$ and $\II$ of $d$-preprojective $\Lambda$-modules and $d$-preinjective $\Lambda$-modules (Definition \ref{def:prep-prei2}).
By the following result, any $d$-RF algebra has a unique $d$-cluster tilting module up to additive equivalence, which is given by $\Pi$.  
For a full subcategory $\XX$ and $\YY$ of an additive category $\CC$, we denote by $\XX\vee\YY$ the full subcategory $\add(\XX\cup\YY)$ of $\CC$.

\begin{proposition}\label{P and I for n-hereditary}
\begin{enumerate}[\rm(a)]
\item \cite[Theorem 1.6]{iya-ct-higher}
If $\Lambda$ is $d$-RF, then $\Pi$ is a $d$-cluster tilting $\Lambda$-module, $\PP=\II=\add\Pi$, and $\UU=\add\{\Pi[di]\mid i\in\Z\}$.
\item \cite[Proposition 4.10(d)]{hio} If $\Lambda$ is $d$-RI, then $\PP=\add\{\nu_d^{-i}(\Lambda)\mid i\ge0\}$, $\II=\add\{\nu_d^i(D\Lambda)\mid i\ge0\}$, and $\UU=\II[-d]\vee\PP$. Moreover, $\Hom_\Lambda(\II,\PP)=0$ and $\PP\cap\II=0$.
\end{enumerate}
\end{proposition}

In the final part of our preparations for this section, we recall the generalization of almost split sequences, or Auslander-Reiten sequences, to $d$-hereditary algebras.

\begin{definition}[{\cite{iya-higher-ar}}]\label{def:nASS}
Let $\CC$ be a Krull-Schmidt $\F$-linear category with Jacobson radical $\rad_\CC$ and let
\begin{equation}\label{d-ass}
Y\xrightarrow{f_d} C_{d-1}\xrightarrow{f_{d-1}} C_{d-2}\xrightarrow{f_{d-2}}\cdots\xrightarrow{f_2}C_1\xrightarrow{f_1}C_0\xrightarrow{f_0} X
\end{equation}
be a complex in $\CC$ where $X$ and $Y$ are indecomposable and each $f_i$ belongs to $\rad_\CC$.
Then we say the sequence \eqref{d-ass} is \emph{$d$-almost split in $\CC$} if both of the following sequences are exact for all objects $M$ in $\CC$:
\begin{eqnarray*}
&0\to\Hom_\CC(M,Y)\arrr{{f_{d}}_*} \Hom_\CC(M,C_{d-1})\arrr{{f_{d-1}}_*} \cdots\arrr{{f_{1}}_*} \Hom_\CC(M,C_0)\arrr{{f_{0}}_*} \rad_{\CC}(M,X)\to 0;&\\
&0\to\Hom_\CC(X,M)\arrr{{f_{0}}^*} \Hom_\CC(C_0,M)\arrr{{f_{1}}^*} \cdots\arrr{{f_{d-1}}^*} \Hom_\CC(C_{d-1},M)\arrr{{f_{d}}^*} \rad_{\CC}(Y,M) \to0.&
\end{eqnarray*}
More generally, we say the sequence \eqref{d-ass} is \emph{weak $d$-almost split in $\CC$} if the above sequences are exact except at $\Hom_\CC(M,Y)$ and $\Hom_\CC(X,M)$ respectively.
\end{definition}

\begin{example}
Let $Q=\left[1\to2\right]$ and $\Lambda=\F Q$.  Then the short exact sequence corresponding to the non-split extension of one simple module by the other is $1$-almost split in $\Lambda\mMod$ but is only weak $1$-almost split in $\Db(\Lambda)$.
\end{example}

It was shown in \cite{hio} (respectively, \cite{iya-higher-ar}) that the category $\PP\vee\II$ has $d$-almost split sequences when $\Lambda$ is $d$-RI (respectively, $d$-RF).
Also it was shown in \cite{iy,io-stab} that $d$-cluster tilting subcategories of triangulated categories have certain analogue of $d$-almost split sequences called AR $(d+2)$-angles.
From these results, one can deduce the following results on $d$-almost split sequences in the category $\UU$, which play a key role in this section.

\begin{theorem}\label{existence of d-ass}
Let $\Lambda$ be a $d$-hereditary algebra.
\begin{enumerate}[\rm(a)]
\item If $\Lambda$ is $d$-RI, then any indecomposable object $X$ (respectively, $Y$) in $\UU$ has a $d$-almost split sequence in $\UU$
\[Y\xrightarrow{f_d} C_{d-1}\xrightarrow{f_{d-1}} C_{d-2}\xrightarrow{f_{d-2}}\cdots\xrightarrow{f_2}C_1\xrightarrow{f_1}C_0\xrightarrow{f_0} X.\]
Moreover, we have $Y\cong\nu_d(X)$ (respectively, $X\cong\nu_d^{-1}(Y)$).
\item If $\Lambda$ is $d$-RF, then any indecomposable object $X$ (respectively, $Y$) in $\UU$ has a weak $d$-almost split sequence in $\UU$
\[Y\xrightarrow{f_d} C_{d-1}\xrightarrow{f_{d-1}} C_{d-2}\xrightarrow{f_{d-2}}\cdots\xrightarrow{f_2}C_1\xrightarrow{f_1}C_0\xrightarrow{f_0} X.\]
Moreover, we have $Y\cong\nu_d(X)$ (respectively, $X\cong\nu_d^{-1}(Y)$), $\Ker(f_d{}_*)=\soc\Hom_{\UU}(-,Y)$ and $\Ker(f_0{}^*)=\soc\Hom_{\UU}(X,-)$.
\end{enumerate}
\end{theorem}

\begin{proof}
In both cases, we only show the assertion for $X$ since the assertion for $Y$ is the dual.

(a) Let $X\in\UU$ be an indecomposable object. 
If $X$ is a projective $\Lambda$-module then $\nu_d^{-1}(X)$ is not projective, as otherwise $X\cong\nu_d\nu_d^{-1}(X)$ would be concentrated in degree $d$ which contradicts our assumption that $\Lambda$ is $d$-RI.
Since $\nu_d:\UU\to\UU$ is an equivalence, it preserves $d$-almost split sequences in $\UU$. Thus we can assume that $X$ is a non-projective object in $\PP$.

It was shown in \cite[Theorem 4.25]{hio} that there exists an exact sequence
\begin{equation}\label{d-ass in P}
0\to Y\xrightarrow{f_d} C_{d-1}\xrightarrow{f_{d-1}} C_{d-2}\xrightarrow{f_{d-2}}\cdots\xrightarrow{f_2}C_1\xrightarrow{f_1}C_0\xrightarrow{f_0} X\to0
\end{equation}
in $\mod\Lambda$ which has terms in $\PP$, $Y=\nu_d(X)$, and gives a $d$-almost split sequence in $\PP\vee\II$. Thus, since Proposition \ref{P and I for n-hereditary}(b) implies $Y\notin\II$, which implies $\rad_\Lambda(Y,\II)=\Hom_\Lambda(Y,\II)$, the following sequences are exact:
\begin{eqnarray*}\label{(P,-)}
&0\to\Hom_\Lambda(\PP,Y)\arrr{{f_{d}}_*} \Hom_\Lambda(\PP,C_{d-1})\arrr{{f_{d-1}}_*} \cdots\arrr{{f_{1}}_*} \Hom_\Lambda(\PP,C_0)\arrr{{f_{0}}_*} \rad_{\Lambda}(\PP,X)\to 0;&\\ \label{(-,I)}
&0\to\Hom_\Lambda(X,\II)\arrr{{f_{0}}^*} \Hom_\Lambda(C_0,\II)\arrr{{f_{1}}^*} \cdots\arrr{{f_{d-1}}^*} \Hom_\Lambda(C_{d-1},\II)\arrr{{f_{d}}^*} \Hom_\Lambda(Y,\II) \to0.&
\end{eqnarray*}
Using Serre duality, we have $\Hom_\Lambda(\PP,\II)=\Hom_{\UU}(\nu^{-1}_d(\II)[-d],\PP)^*$. 
As $\Lambda$ is $d$-RI, we have $\II\subseteq\nu^{-1}_d(\II)$ by Proposition \ref{P and I for n-hereditary}(b).
Therefore, the latter exact sequence gives an exact sequence
\begin{eqnarray*}\label{(I[-d],-)}
0\to\Hom_\UU(\II[-d],Y)\arrr{{f_{d}}_*} \Hom_\UU(\II[-d],C_{d-1})\arrr{{f_{d-1}}_*}\cdots&\arrr{{f_{1}}_*}&\Hom_\UU(\II[-d],C_0)\\
&\arrr{{f_{0}}_*}&\Hom_\UU(\II[-d],X)\to 0.
\end{eqnarray*}
Since $\UU=\II[-d]\vee\PP$ by \cite[Proposition 4.10(c)]{hio}, the above exact sequences give an exact sequence
\[0\to\Hom_{\UU}(\UU,Y)\arrr{{f_{d}}_*} \Hom_\UU(\UU,C_{d-1})\arrr{{f_{d-1}}_*} \cdots\arrr{{f_{1}}_*} \Hom_\UU(\UU,C_0)\arrr{{f_{0}}_*} \rad_{\UU}(\UU,X)\to 0.\]
Dually, the following sequence is exact.
\[0\to\Hom_\UU(X,\UU)\arrr{{f_{0}}^*} \Hom_\UU(C_0,\UU)\arrr{{f_{1}}^*} \cdots\arrr{{f_{d-1}}^*} \Hom_\UU(C_{d-1},\UU)\arrr{{f_{d}}^*} \rad_{\UU}(Y,\UU) \to0.\]
Thus the sequence \eqref{d-ass in P} is a $d$-almost split sequence in $\UU$.

(b) 
By \cite[Theorem 1.23]{iya-ct-higher}, $\UU$ is a $d$-cluster tilting subcategory of $\Db(\Lambda)$. By \cite[Theorem 3.10]{iy}, there exist triangles
\[X_{i+1}\arrr{h_{i+1}} C_i\arrr{{g_i}} X_i\to X_{i+1}[1]\]
in $\Db(\Lambda)$ for $0\le i\le d-1$ satisfying the following conditions:
\begin{enumerate}[$\bullet$]
\item $X_0=X$, $X_{d}=\nu_d(X)$, and $C_i\in\UU$ for any $0\le i\le d-1$;
\item $\Hom_\UU(\UU,C_0)\arrr{{g_{0}}_*} \rad_{\UU}(\UU,X)\to 0$ and $\Hom_\UU(C_{d-1},\UU)\arrr{{h_{d}}^*} \rad_{\UU}(\nu_d(X),\UU)\to 0$ are exact. 
\end{enumerate}
Let $f_d:=h_d$, $f_i:=h_ig_i$ and $f_0:=g_0$. Then we have a complex
\[\nu_d(X)\arrr{f_d} C_{d-1}\arrr{f_{d-1}} C_{d-2}\arrr{f_{d-2}}\cdots\arrr{f_2} C_1\arrr{f_1} C_0\arrr{f_0} X.\]
Moreover, as $\Lambda$ is $d$-RF, $\nu(\UU)=\UU$ by \cite[Theorem 3.1(1)$\Rightarrow$(3)]{io-stab} and hence $\UU[d]=\UU$. So, by \cite[Lemma 4.3]{io-stab}, we have an exact sequence 
\begin{eqnarray*}
\cdots\to\Hom_\UU(\UU,C_0[-d])\arrr{{f_{0}[-d]}_*}\Hom_\UU(\UU,X[-d])\to\\
\Hom_\UU(\UU,\nu_d(X))\arrr{{f_{d}}_*} \Hom_\UU(\UU,C_{d-1})\arrr{{f_{d-1}}_*} \cdots\arrr{{f_{1}}_*} \Hom_\UU(\UU,C_0)\arrr{{f_{0}}_*} \rad_{\UU}(\UU,X)\to0.
\end{eqnarray*}
Thus $\Cok(f_{0}{}_*:\Hom_\UU(-,C_0)\to \Hom_{\UU}(-,X))$ is a simple $\UU$-module, and hence $\Ker(f_d{}_*)=\Cok(f_{0}[-d]{}_*)$ is a simple $\UU$-module since $[d]:\UU\to\UU$ is an autoequivalence.
Because $X[-d]\in\UU$ is indecomposable, $\Hom_{\UU}(X[-d],-)$ is an indecomposable projective functor and thus it has a simple top.  Hence the $\UU$-module $\Hom_{\UU}(-,\nu_d(X))\cong\Hom_{\UU}(X[-d],-)^*$ has a simple socle.  Therefore $\Ker(f_d{}_*)=\soc\Hom_{\UU}(\UU,\nu_d(X))$.

Dually, we have an exact sequence
\[\Hom_\UU(X,\UU)\arrr{{f_{0}}^*} \Hom_\UU(C_0,\UU)\arrr{{f_{1}}^*} \cdots\arrr{{f_{d-1}}^*} \Hom_\UU(C_{d-1},\UU)\arrr{{f_{d}}^*} \rad_{\UU}(Y,\UU) \to0\]
such that $\Ker({f_{0}}^*)=\soc\Hom_\UU(X,\UU)$. Thus the assertions hold.
\end{proof}

\subsection{Resolutions of simple modules over higher preprojective algebras}

For the rest of this section, $\Lambda$ is a $d$-hereditary algebra and $\Pi$ is its higher preprojective algebra. We will assume that $\Lambda$ is basic and ring-indecomposable.
We regard $\Pi$ as a $\Z$-graded algebra with the tensor grading. Then we have an isomorphism
\[\Pi\cong\bigoplus_{i\in\Z}\Hom_{\Db(\Lambda)}(\Lambda,\nu_d^{-i}(\Lambda))\]
of $\Z$-graded algebras.

For a group $\Psi$ and a $\Psi$-graded ring $\Gamma$, we denote by $\Gamma\mGgrmod$ (respectively, $\Gamma\mGgrproj$) the category of finitely generated (respectively, finitely generated projective) $\Psi$-graded $\Gamma$-modules.
We start with the following easy observation.
\begin{lemma}\label{Gproj}
Let $\CC$ be an additive category and $\Psi$ a group acting on $\CC$.
Assume that $M\in\CC$ is an object satisfying $\CC=\add\{\psi M\mid \psi\in \Psi\}$.
Define a $\Psi$-graded ring by $\Gamma:=\bigoplus_{\psi\in \Psi}\Hom_{\CC}(M,\psi M)$.
Then there are equivalences of additive categories
\begin{eqnarray*}
\bigoplus_{\psi\in \Psi}\Hom_{\CC}(M,\psi -):\CC\to\Ggrprojm\Gamma\ \mbox{ and }\ 
\bigoplus_{\psi\in \Psi}\Hom_{\CC}(-,\psi M):\CC\to\Gamma\mGgrproj.
\end{eqnarray*}
\end{lemma}

Applying Lemma \ref{Gproj} to the category $\UU$ and the group $\{\nu_d^{-i}\mid i\in\Z\}\cong\Z$, we have the following description of the category $\UU$.

\begin{proposition}\label{Zproj}
\begin{enumerate}[\rm(a)]
\item There are equivalences of additive categories
\begin{eqnarray*}
&G:=\bigoplus_{i\in\Z}\Hom_{\Db(\Lambda)}(\Lambda,\nu_d^{-i}(-)):\UU\to\grprojm\Pi,&\\
&H:=\bigoplus_{i\in\Z}\Hom_{\Db(\Lambda)}(-,\nu_d^{-i}(\Lambda)):\UU\to\Pi\mgrproj.&
\end{eqnarray*}
In particular, there are equivalences of additive categories
\begin{eqnarray*}
G_*:\grModm\Pi\to\Modm\UU\ \mbox{ and }\ H_*:\Pi\mgrMod\to\UU\mMod.
\end{eqnarray*}
\item The following diagram commutes up to natural isomorphism.
\[\xymatrix@R=2em{
\UU\ar[rr]^G\ar@{=}[d]&&\grprojm\Pi\ar@<.2em>[d]^{\Hom_{\Pi^{\op}}(-,\Pi)}\\
\UU\ar[rr]_H&&\Pi\mgrproj\ar@<.2em>[u]^{\Hom_\Pi(-,\Pi)}.
}\]
\end{enumerate}
\end{proposition}

Now we are ready to state the main result of this subsection.
It asserts that minimal projective resolutions of $\Z$-graded simple modules over the higher preprojective algebra $\Pi$ of a $d$-hereditary algebra $\Lambda$ are induced from $d$-almost split sequences in $\UU$.

\begin{theorem}\label{d-ass gives projective resolution}
Let $X$ be an indecomposable object in $\UU$, and
\[Y\xrightarrow{f_d} C_{d-1}\xrightarrow{f_{d-1}} C_{d-2}\xrightarrow{f_{d-2}}\cdots\xrightarrow{f_2}C_1\xrightarrow{f_1}C_0\xrightarrow{f_0} X\]
a $d$-almost split sequence in $\UU$.
\begin{enumerate}[\rm(a)]
\item There exist exact sequences
\begin{eqnarray*}
&GY\xrightarrow{Gf_d} GC_{d-1}\xrightarrow{Gf_{d-1}}\cdots\xrightarrow{Gf_2} GC_1\xrightarrow{Gf_1} GC_0\xrightarrow{Gf_0} GX\to T\to0&\\
&HX\xrightarrow{Hf_0} HC_0\xrightarrow{Hf_1}HC_1\xrightarrow{Hf_2}\cdots\xrightarrow{Hf_{d-1}} HC_{d-1}\xrightarrow{Hf_d} HY\to U\to0&
\end{eqnarray*}
in $\grModm\Pi$ and $\Pi\mgrMod$, where $T$ and $U$ are simple.
\item If $\Lambda$ is $d$-RI, then $Gf_d$ and $Hf_0$ are monomorphisms.
\item If $\Lambda$ is $d$-RF, then $\Ker Gf_d=\soc GY$ and $\Ker Hf_0=\soc HX$. Moreover these are simple.
\end{enumerate}
\end{theorem}

\begin{proof}
(a) By Theorem \ref{existence of d-ass}(a), we have an exact sequence
\begin{eqnarray*}
\bigoplus_{i\in\Z}\Hom_{\UU}(\nu_d^i(\Lambda),Y)\xrightarrow{f_d{}_*}\bigoplus_{i\in\Z}\Hom_{\UU}(\nu_d^i(\Lambda),C_{d-1})\xrightarrow{f_{d-1}{}_*}\cdots&\xrightarrow{f_1{}_*}&\bigoplus_{i\in\Z}\Hom_{\UU}(\nu_d^i(\Lambda),C_0)\\
&\xrightarrow{f_0{}_*}&\bigoplus_{i\in\Z}\rad_{\UU}(\nu_d^i(\Lambda),X)\to0.
\end{eqnarray*}
This gives the first sequence. Dually, we obtain the second sequence.
It follows from Proposition \ref{Zproj}(a) that $T$ and $U$ are simple.

(b)(c) These follow from Theorem \ref{existence of d-ass}(a)(b). {It follows from Proposition \ref{Zproj}(a) that $\Ker Gf_d$ and $\Ker Hf_0$ are simple if $\Lambda$ is $d$-RF.}  
\end{proof}

We say that an algebra $A$ is \emph{twisted-periodic} if, for some $i\geq1$, $\Omega_{A^\en}^i(A)\cong A_\sigma$ as $A^\en$-modules for some $\sigma\in\Aut(A)$, i.e., the projective resolution of the identity bimodule is periodic up to a twist by some algebra automorphism.

As an application of our results, we have the following result for $d$-RF case. The selfinjectivity was first proved in \cite{io-stab}, and the twisted-periodicity was first proved by Dugas \cite{dug}.

\begin{corollary}\label{cor:nrf-pi-twper}
Let $\Lambda$ be a $d$-RF algebra and $\Pi$ its $(d+1)$-preprojective algebra.
\begin{enumerate}[\rm(a)]
\item $\Pi$ is self-injective.
\item $\Pi$ is twisted-periodic of period $d+2$. 
\end{enumerate}
\end{corollary}

\begin{proof}
(a) It follows from Theorem \ref{d-ass gives projective resolution} that $\Ext^i_{\Pi}(T,\Pi)=0$ holds for any $\Z$-graded simple $\Pi$-modules and $0<i<d+1$.
Thus $\Ext^1_{\Pi}(-,\Pi)=0$ holds on $\mod\Pi$, and therefore $\Pi$ is injective as a $\Pi$-module.

(b) Since $\Lambda$ is a factor algebra of $\Pi$ by the ideal $\bigoplus_{i>0}\Pi_i$ contained in the radical, each simple $\Pi$-module $S$ is realized as the top of $GP$, where $P$ is an indecomposable projective $\Lambda$-module.
Thus, by Theorem \ref{d-ass gives projective resolution}(c), the sum $S=\bigoplus S_i$ of the simple $\Pi$-modules is periodic of period $d+2$.
This implies the assertion by \cite[Theorem 1.4]{gss}.
\end{proof}

We note that the twisted-periodicity is closely related to the stably Calabi-Yau property
(e.g.\ \cite[Theorem 1.8]{iv}). In fact, $\Pi$ is known to be stably $(d+1)$-Calabi-Yau \cite[Theorem 1.1(a)]{io-stab}.

As another application our results, we have the following result for $d$-RI case.  

\begin{corollary}\label{cor:gldimri}
Let $\Lambda$ be a $d$-RI algebra and $\Pi$ its $(d+1)$-preprojective algebra.
\begin{enumerate}[\rm(a)]
\item $\Pi$ has left and right global dimension $d+1$ (c.f.\ Appendix A).
\item Any $\Z$-graded simple right $\Pi$-module $T$ satisfies
\[\Ext^i_{\Pi^{\op}}(T,\Pi) \cong
\begin{cases}
T^*\grsh{1} & \text{if } i=d+1;\\
0 & \text{otherwise.}
\end{cases}\]
\item Any $\Z$-graded simple left $\Pi$-module $U$ satisfies
\[\Ext^i_\Pi(U,\Pi) \cong
\begin{cases}
U^*\grsh{1} & \text{if } i=d+1;\\
0 & \text{otherwise.}
\end{cases}\]
\end{enumerate}
\end{corollary}

\begin{proof}
It follows from Theorem \ref{d-ass gives projective resolution} that any $\Z$-graded simple $\Pi^{\op}$-module $T$ has projective dimension $d+1$ and satisfies the equalities of extension groups. Thus (b) holds, and dually (c) holds. They imply (a) by Theorem \ref{global dimension}.
\end{proof}

Corollary \ref{cor:gldimri} says that $\Pi$, with the tensor grading, is a generalized Artin-Schelter regular algebra of dimension $d+1$ and Gorenstein parameter $1$ in the sense of \cite{mv-serre,ms11,mm,rr} (see also \cite{as}).  This is equivalent to a result \cite[Theorem 4.2]{mm} of Minamoto-Mori up to  \cite[Theorem 5.2]{rr}, and also can be deduced from results of Keller \cite{kel-dcy}.

\subsection{$\Z^2$-graded higher preprojective algebras}

Here, we consider gradings on higher preprojective algebras, which will be used in Section \ref{Koszul higher}.
Let $\Lambda$ be a positively $\Z$-graded algebra
$$\Lambda =\bigoplus_{i\in\Z}\Lambda_i$$
with radical grading (see Section \ref{subsec:kos}).
The enveloping algebra $\Lambda^\en$ of $\Lambda$ has a $\Z$-grading given by $$(\Lambda^\en)_i=\bigoplus_{i=j+k}\Lambda_j\otimes_\F\Lambda_k.$$
Using the $\Z$-grading on $\Lambda$, we define a new $\Z$-grading on the higher preprojective algebra $\Pi$.

For $i>0$ and finitely generated $\Z$-graded $\Lambda$-modules $M$ and $N$, let $\ext^i_\Lambda(M,N)$ denote the $\Z$-graded $i$th ext space (our notation follows \cite[Section 2.1]{bgs}). 
Then we have an equality
$$\Ext^i_\Lambda(M,N)=\bigoplus_{j\in\Z}\ext^i_\Lambda(M,N\grsh{j}).$$
Hence $\Ext^i_\Lambda(M,N)$ has a $\Z$-grading whose degree $j$ part is $\ext^i_\Lambda(M,N\grsh{j})$.

Now we define the $\Z$-grading on the $\Lambda^\en$-module $E=\Ext^d_\Lambda(\Lambda^*,\Lambda)$ by
\begin{equation}\label{grading on E}
E=\bigoplus_{j\in \Z}\ext^d_\Lambda(\Lambda^*,\Lambda\grsh{j}).
\end{equation}
Then, as in Lemma \ref{lem:alt-E}, we can show that there are isomorphisms
$$E\cong\bigoplus_{j\in \Z}\ext^d_{\Lambda^\en}(\Lambda,\Lambda^\en\grsh{j})\cong\bigoplus_{j\in \Z}\ext^d_{\Lambda^\op}(\Lambda^*,\Lambda\grsh{j})$$
of $\Z$-graded $\Lambda^\en$-modules.
Let $\Lambda\mgrMod$ denote the category of finitely generated $\Z$-graded left $\Lambda$-modules.
We lift the functors $\tau_d$ and $\tau_d^-$ to $\Z$-graded $\Lambda$-modules as follows.
\[\tau_d:=\Hom_\Lambda(E,-):\Lambda\mgrMod\to\Lambda\mgrMod
\ \mbox{ and }\ \tau_d^-:=E\otimes_\Lambda-:\Lambda\mgrMod\to\Lambda\mgrMod.\]

\begin{definition}\label{def-grpreproj}
\begin{enumerate}[\rm(a)]
\item The \emph{$\Z^2$-graded $(d+1)$-preprojective algebra} of a $\Z$-graded algebra $\Lambda =\bigoplus_{i\in\Z}\Lambda_i$ is the tensor algebra of the $\Z$-graded $\Lambda^\en$-module $E$:
$$\Pi(\Lambda)=\Tens_\Lambda(E).$$
The first part of the $\Z^2$-grading is the tensor grading (Definition \ref{ungraded preprojective algebra}). 
The second part of the $\Z^2$-grading is called the \emph{$\Lambda$-grading}, which is a natural grading on $E^i$ for any $i\ge0$ given by the $\Z$-grading on $E$ in \eqref{grading on E}.
\item We consider a single $\Z$-grading on $\Pi$, called the \emph{$(d+1)$-total grading}, by defining
$$\Pi_\ell:=\bigoplus_{(d+1)i+j=\ell}\Pi_{i,j}$$
where $\Pi_{i,j}=(E^i)_j$ denotes the $j$th graded component of $E^i$.
\end{enumerate}
\end{definition}

Later we will use the following observation.

\begin{proposition}\label{prop:e-gen-pos-deg}
If $\Lambda$ is Koszul, then $E$ is generated in degree $-d$. Therefore the $(d+1)$-total grading of $\Pi$ gives a radical grading. 
\end{proposition}

\begin{proof}
If $\Lambda$ is Koszul, then $P_d$ is generated in degree $d$ by Theorem \ref{characterize koszul}, and the former assertion follows.  Since $\Pi_0=\Lambda_0$, $\Pi_1=\Lambda_1\oplus E_{-d}$ and $E_{-d}=\head_{\Lambda^\en}E$, the latter assertion follows.
\end{proof}

\subsection{Koszul properties of higher preprojective algebras}\label{Koszul higher}

Let $\Lambda$ be a $d$-hereditary $\F$-algebra. 
In this section, we further assume that $\Lambda$ is a $\Z$-graded algebra $\Lambda=\bigoplus_{i\in\Z}\Lambda_i$. 
We denote by $\Db(\Lambda\mgrMod)$ the bounded derived category of $\Lambda\mgrMod$.
As in the ungraded case, we define an autoequivalence
\begin{eqnarray*}
\nu_d=\Lambda^*[-d]\dert_\Lambda-:\Db(\Lambda\mgrMod)\to\Db(\Lambda\mgrMod)
\end{eqnarray*}
and a full subcategory
\[\UU^{\Z}:=\add\{\nu_d^{-i}(\Lambda)(j)\mid i,j\in\Z\}\subseteq\Db(\Lambda\mgrMod).\]
We have the following graded version of Theorem \ref{existence of d-ass}.

\begin{theorem}\label{existence of d-ass2}
Let $\Lambda$ be a $\Z$-graded $d$-hereditary algebra. 
\begin{enumerate}[\rm(a)]
\item If $\Lambda$ is $d$-RI, then any indecomposable object $X$ (respectively, $Y$) in $\UU^{\Z}$ has a $d$-almost split sequence in $\UU^{\Z}$
\[Y\xrightarrow{f_d} C_{d-1}\xrightarrow{f_{d-1}} C_{d-2}\xrightarrow{f_{d-2}}\cdots\xrightarrow{f_2}C_1\xrightarrow{f_1}C_0\xrightarrow{f_0} X.\]
Moreover, we have $Y\cong\nu_d(X)$ (respectively, $X\cong\nu_d^{-1}(Y)$).
\item If $\Lambda$ is $d$-RF, then any indecomposable object $X$ (respectively, $Y$) in $\UU^{\Z}$ has a weak $d$-almost split sequence in $\UU^{\Z}$
\[Y\xrightarrow{f_d} C_{d-1}\xrightarrow{f_{d-1}} C_{d-2}\xrightarrow{f_{d-2}}\cdots\xrightarrow{f_2}C_1\xrightarrow{f_1}C_0\xrightarrow{f_0} X.\]
Moreover, we have $Y\cong\nu_d(X)$ (respectively, $X\cong\nu_d^{-1}(Y)$), $\Ker(f_d{}_*)=\soc\Hom_{\UU}(-,Y)$ and $\Ker(f_0{}^*)=\soc\Hom_{\UU}(X,-)$.
\end{enumerate}
\end{theorem}

\begin{proof}
The proof is very similar to Theorem \ref{existence of d-ass}.
\end{proof}

Let $\Pi$ be the $\Z^2$-graded $(d+1)$-preprojective algebra. Recall from Definition \ref{def-grpreproj} that the first entry of the $\Z^2$-grading is the tensor grading, and the second one is the $\Lambda$-grading. 

On the other hand, we consider the action of $\Z^2$ on $\UU^{\Z}$ given by $(i,j)\mapsto\nu_d^{-i}(j)$.
The following description of the category $\UU^{\Z}$ follows directly from Lemma \ref{Gproj} and the definition.

\begin{proposition}\label{Z*Zproj}
\begin{enumerate}[\rm(a)]
\item There are equivalences of additive categories
\begin{eqnarray*}
&G^{\Z}=\bigoplus_{i,j\in\Z}\Hom_{\Db(\Lambda\mgrMod)}(\Lambda,\nu_d^{-i}(-)(j)):\UU^{\Z}\cong\bigrprojm\Pi,&\\
&H^{\Z}=\bigoplus_{i,j\in\Z}\Hom_{\Db(\Lambda\mgrMod)}(-,\nu_d^{-i}(\Lambda)(j)):\UU^{\Z}\cong\Pi\mbigrproj.&
\end{eqnarray*}
\item The following diagram commutes up to natural isomorphism.
\[\xymatrix@R=2em{
\UU^\Z\ar[rr]^{G^{\Z}}\ar@{=}[d]&&\bigrprojm\Pi\ar@<.2em>[d]^{\Hom_{\Pi^{\op}}(-,\Pi)}\\
\UU^\Z\ar[rr]_{H^{\Z}}&&\Pi\mbigrproj\ar@<.2em>[u]^{\Hom_\Pi(-,\Pi)}.
}\]
\item We have the following commutative diagrams.
\begin{equation*}
\xymatrix@R=2em{\UU^{\Z}\ar[rr]^{G^{\Z}}\ar[d]^{\nu_d^{-1}}&&\bigrprojm\Pi\ar[d]^{(1,0)}\\
\UU^{\Z}\ar[rr]^{G^{\Z}}&&\bigrprojm\Pi}\ \ \ \ \ 
\xymatrix@R=2em{\UU^{\Z}\ar[rr]^{H^{\Z}}\ar[d]^{\nu_d}&&\Pi\mbigrproj\ar[d]^{(1,0)}\\
\UU^{\Z}\ar[rr]^{H^{\Z}}&&\Pi\mbigrproj}
\end{equation*}
\end{enumerate}
\end{proposition}

Immediately, we have the following $\Z$-graded version of Theorem \ref{d-ass gives projective resolution}.

\begin{theorem}\label{d-ass gives projective resolution2}
Let $\Lambda$ be a $\Z$-graded $d$-hereditary algebra.
For an indecomposable object $X$ in $\UU^{\Z}$, we consider a $d$-almost split sequence in $\UU^{\Z}$:
\[Y\xrightarrow{f_d} C_{d-1}\xrightarrow{f_{d-1}}C_{d-2}\xrightarrow{f_{d-2}}\cdots\xrightarrow{f_2} C_1\xrightarrow{f_1} C_0\xrightarrow{f_0} X.\]
\begin{enumerate}[\rm(a)]
\item There exist exact sequences
\begin{eqnarray*}
&G^{\Z}Y\xrightarrow{G^{\Z}f_d} G^{\Z}C_{d-1}\xrightarrow{G^{\Z}f_{d-1}}\cdots\xrightarrow{G^{\Z}f_2} G^{\Z}C_1\xrightarrow{G^{\Z}f_1} G^{\Z}C_0\xrightarrow{G^{\Z}f_0} G^{\Z}X\to S\to0&\\
&H^{\Z}X\xrightarrow{H^{\Z}f_0} H^{\Z}C_0\xrightarrow{H^{\Z}f_1}H^{\Z}C_1\xrightarrow{H^{\Z}f_2}\cdots\xrightarrow{H^{\Z}f_{d-1}} H^{\Z}C_{d-1}\xrightarrow{H^{\Z}f_d} H^{\Z}Y\to T\to0&
\end{eqnarray*}
in $\bigrModm\Pi$ and $\Pi\mbigrMod$, where $S$ and $T$ are simple.
\item If $\Lambda$ is $d$-RI, then $G^{\Z}f_d$ and $H^{\Z}f_0$ are monomorphisms.
\item If $\Lambda$ is $d$-RF, then $\Ker G^{\Z}f_d=\soc G^{\Z}Y$ and $\Ker H^{\Z}f_0=\soc H^{\Z}X$ hold. Moreover these are simple.
\end{enumerate}
\end{theorem}

\begin{proof}
The assertions follow from Theorem \ref{existence of d-ass2} by a similar argument to the proof of Theorem \ref{d-ass gives projective resolution}.
\end{proof}

In the rest of this section, we further assume that $\Lambda=\bigoplus_{i\ge0}\Lambda_i=\Tens_S(V)/I$ is a Koszul algebra and $S=\Lambda_0$ is a semisimple $\F$-algebra.

We now recall the theory of almost Koszul duality due to Brenner, Butler, and King \cite{bbk}.
Let $S$ be a semisimple finite-dimensional $\F$-algebra and $A=\bigoplus_{i\geq0}A_i$ a nonnegatively $\Z$-graded $S$-algebra with $A_0=S$.

\begin{definition}\label{define almost Koszul}
The $\Z$-graded algebra $A$ is \emph{almost Koszul}, or \emph{$(p,q)$-Koszul}, if there exist integers $p,q\geq1$ such that $A_i=0$ for all $i>p$ and there is an exact sequence
$$0\to S'\to P_q\to\dots\to P_0\to S\to0$$
of $\Z$-graded $A$-modules with projective $A$-modules $P_i$ generated in degree $i$ and a semisimple $A$-module $S'$ concentrated in degree $p+q$.
\end{definition}
\noindent
Note that it does not matter whether we consider left or right $A$-modules \cite[Proposition 3.4]{bbk}.

\begin{theorem}\label{thm:kos-alm-kos}
Let $\Lambda$ be a Koszul $d$-hereditary algebra, and $\Pi$ its $(d+1)$-preprojective algebra with the $(d+1)$-total grading given in Definition \ref{def-grpreproj}.
\begin{enumerate}[\rm(a)]
\item If $\Lambda$ is $d$-RI, then $\Pi$ is Koszul.
\item If $\Lambda$ is $d$-RF, then $\Pi$ is almost Koszul. It is $(p,d+1)$-Koszul, where $p=\max\{i\ge0\mid\Pi_i\neq0\}$ with respect to the total grading.
\end{enumerate}
\end{theorem}
\begin{proof}
Let $\grModm\Pi$ be the category of $\Z$-graded $\Pi$-modules with respect to the $(d+1)$-total grading on $\Pi$.
Let $S$ be a $\Z$-graded simple $\Pi$-module $S$ concentrated in degree 0.
Consider the functor $F:\bigrModm\Pi\to\grModm\Pi$ given by $\bigoplus_{(i,j)\in\Z^2}X_{i,j}\mapsto \bigoplus_{\ell\in\Z}X_\ell$, where $X_\ell=\bigoplus_{(d+1)i+j=\ell}X_{i,j}$.
Let $G'=F\circ G^{\Z}$ and $H'=F\circ H^{\Z}$. 
Then Theorem \ref{d-ass gives projective resolution2}(a) gives the first $d+1$ terms of minimal $\Z$-graded projective resolution
\begin{eqnarray}\label{GZ sequence}
G'Y\xrightarrow{G'f_d} G'C_{d-1}\xrightarrow{G'f_{d-1}}\cdots\xrightarrow{G'f_2} G'C_1\xrightarrow{G'f_1} G'C_0\xrightarrow{G'f_0} G'X\to S\to0
\end{eqnarray}
and the exact sequence
\begin{eqnarray}\label{HZ sequence}
H'X\xrightarrow{H'f_0} H'C_0\xrightarrow{H'f_1}H'C_1\xrightarrow{H'f_2}\cdots\xrightarrow{H'f_{d-1}} H'C_{d-1}\xrightarrow{H'f_d} H'Y\to T\to0.
\end{eqnarray}
To prove both assertions, we only have to show that $G'C_i$ is generated in degree $i+1$.
Since $\Lambda$ is Koszul, by Proposition \ref{prop:e-gen-pos-deg}, the $(d+1)$-total grading and the radical grading on $\Pi$ agree.
Since $G'X$ is generated in degree $0$ and \eqref{GZ sequence} is minimal, $G'C_i$ is generated in degrees at least $i+1$.

By Proposition \ref{Z*Zproj}(c), we have $G^{\Z}Y=G^{\Z}\nu_d(X)=(G^{\Z}X)(-1,0)$ and hence $G'Y=(G'X)(-d-1)$.
Thus $G'Y$ is generated in degree $d+1$, and hence $H'Y$ is generated in degree $-d-1$ by Proposition \ref{Z*Zproj}(b).
Since \eqref{HZ sequence} is minimal, $H'C_i$ is generated in degrees at least $-i-1$ and hence $G'C_i$ is generated in degrees at most $i+1$. Thus the assertion follows.
\end{proof}

\section{Quadratic duals of higher preprojective algebras}

The aim of this section is to compare the quadratic duals of the higher preprojective algebras and certain twisted trivial extension algebras of the quadratic duals for Koszul algebras.

\subsection{Graded trivial extension algebras}\label{ss:gtea}

For any finite dimensional $\F$-algebra $\Gamma$, there is a well-known way to construct a new algebra called the trivial extension algebra.  We describe a graded version of this, which can be seen as an extension of $\Gamma$ by a twist of the dual bimodule $\Gamma^*$.
\begin{definition}
Let $\Gamma$ be a non-negatively $\Z$-graded finite-dimensional algebra and $n\in\Z$.
The \emph{graded $(d+1)$-trivial extension algebra of $\Gamma$}, denoted $\Triv_{d+1}(\Gamma)$, is the $\Z$-graded vector space $\Gamma\oplus \Gamma^*\grsh{-d-1}$ with multiplication given by 
\[ (a,f)\cdot (b,g)=(ab,ag+(-1)^{di}fb)\]
when $b\in\Gamma_i$ is a homogeneous element of degree $i$.
\end{definition}
We have used the fact that $\Gamma$, and hence $\Gamma^*$, has a natural structure of a $\Gamma^\en$-module.

One can interpret $\Z$-graded $d$-trivial extensions in the following way.  First, let $\sigma:\Gamma\to\Gamma$ be the algebra automorphism defined by $\sigma(a)=(-1)^ia$ for $a\in\Gamma_i$.  Then $\Triv_{d+1}(\Gamma)$ is the trivial extension of $\Gamma$ by the twisted bimodule ${}_{\sigma^d}\Gamma^*$.
Note that another multiplication rule $(a,f)\cdot (b,g)=(ab,(-1)^{di}ag+fb))$ with $a\in\Gamma_i$ used in \cite{g-zz} gives an isomorphic $\Z$-graded algebra.

In the rest of this section, we assume that
\begin{equation}\label{presentation of Lambda}
\Lambda=\Tens_S(V)/(R)
\end{equation}
is a Koszul algebra with a separable $\F$-algebra $S$, and $\Gamma$ is its quadratic dual $\Gamma=\Lambda^!$.  Recall that we have $S^\en$-modules $K_i$ with $K_0=S$, $K_1=V$ and $K_2=R$ and maps $\iota^\ell_i:K_i\into V\otimes_SK_{i-1}$ and $\iota^r_i:K_i\into K_{i-1}\otimes_SV$.  
By Lemma \ref{Koszul dual is K*}, we have an isomorphism of $\Z$-graded algebras
$$\Lambda^!=\bigoplus_{i\geq 0}(\Lambda^!)_i\cong\bigoplus_{i\geq0}K_i^{*\ell}$$
where the algebra structure on $\bigoplus_{i\geq0}K_i^{*\ell}$ is given by $(\iota^\ell_i)^{*\ell}:K_{i-1}^{*\ell}\otimes_SV^{*\ell}\to K_i^{*\ell}$ and $(\iota^r_i)^{*\ell}:V^{*\ell}\otimes_SK_{i-1}^{*\ell}\to K_i^{*\ell}$.
Since $(\Lambda^!)_i=\Ext^i_\Lambda(S,S)$, the global dimension $d$ of $\Lambda$ is the maximal $i$ such that $(\Lambda^!)_i\neq0$, and we have
\begin{equation*}\label{eq:graded-parts-of-triv}
\Triv_{d+1}(\Lambda^!)_i=K_i^{*\ell}\oplus K_{d+1-i}^{*\ell*}
\end{equation*}
where $K_i=0$ for $i<0$ or $i>n$, and $\Triv_{d+1}(\Lambda^!)$ is concentrated in degrees $0$ to $d+1$.

Recall from Proposition \ref{prop:pi-quadratic} that, if $\Lambda$ is a Koszul algebra, then its higher preprojective algebra $\Pi$ is quadratic.
We are now able to state the following result for the quadratic dual $\Pi^!$ of $\Pi$.

\begin{theorem}\label{thm:canonical map}
Let $\Lambda$ be a finite dimensional Koszul $\F$-algebra of global dimension $d$ such that $S=\Lambda_0$ is a separable $\F$-algebra, and let $\Pi$ be its higher preprojective algebra with radical grading.
\begin{enumerate}[\rm(a)]
\item There exists a morphism $\phi:\Pi^!\to\Triv_{d+1}(\Lambda^!)$ of $\Z$-graded $\F$-algebras, which is an isomorphism in degrees $0$ and $1$ and is injective in degree $2$.
\item $\phi$ is surjective if and only if $(\Lambda^!)_d=\soc_{\Lambda^{!\en}}(\Lambda^!)$.  In this case $\phi$ is an isomorphism in degrees $0$, $1$ and $2$.
\item $\phi$ is an isomorphism if and only if $(\Lambda^!)_d=\soc_{\Lambda^{!\en}}(\Lambda^!)$ holds and $\Triv_{d+1}(\Lambda^!)$ is quadratic.
\end{enumerate}
\end{theorem}

To prove this, we need the following technical observation. Consider the $\Z$-graded $\Lambda^{!\en}$-module
\[L:=\bigoplus_{i\in\Z}K_{d+1-i}^{\vee*\ell}\]
whose structure is given by $(\theta_i^\ell)^{*\ell}:V^{*\ell}\otimes_SK_{i-1}^{\vee*\ell}\to K_i^{\vee*\ell}$ and $(\theta_i^r)^{*\ell}:K_{i-1}^{\vee*\ell}\otimes_SV^{*\ell}\to K_i^{\vee*\ell}$ obtained from \eqref{theta}.

\begin{lemma}\label{Lambda!*}
We have an isomorphism $\Lambda^{!*}\cong L$ of $\Z$-graded $\Lambda^{!\en}$-modules.
\end{lemma}

\begin{proof}
Applying Lemma \ref{duals} and its dual to the $\Z$-graded $\Lambda^{!\en}$-module $\bigoplus_{i\in\Z}K_i^{*\ell}$, we obtain isomorphisms of $\Z$-graded $\Lambda^{!\en}$-modules
$\Lambda^{!*}=\bigoplus_{i\in\Z}K_i^{*\ell*}\cong\bigoplus_{i\in\Z}K_i^{*\ell*r}\cong\bigoplus_{i\in\Z}K_i$.
Similarly we obtain isomorphisms of $\Z$-graded $\Lambda^{!\en}$-modules $L=\bigoplus_{i\in\Z}K_i^{\vee*\ell}\cong\bigoplus_{i\in\Z}K_i^{\vee\vee}\cong\bigoplus_{i\in\Z}K_i$.
Thus the assertion follows.
\end{proof}

We are ready to prove Theorem \ref{thm:canonical map}.

\begin{proof}[Proof of Theorem \ref{thm:canonical map}]
Twisting the right action of $\Lambda^!$ on $L$ as $f\cdot a:=(-1)^{di}fa$ for $f\in L$ and $a\in(\Lambda^!)_i$, we obtain an $\Lambda^{!\en}$-module $L'$.
Thanks to Lemma \ref{Lambda!*}, we can regard $\triv_{d+1}(\Lambda^!)$ as $\triv_{d+1}(\Lambda^!)=\Lambda^!\oplus L'=\bigoplus_{i\in\Z}(K_i^{*\ell}\oplus K_{d+1-i}^{\vee*\ell})$.

(a) By Proposition \ref{prop:piquiv}, $\Pi$ is a quotient of $\Tens_S(\overline{V})$, so $\Pi^!$ is a quotient of $\Tens_S(\overline{V}^{*\ell})$. Since
\[\Tens_S(\overline{V}^{*\ell})_0=S=\triv_{d+1}(\Lambda^!)_0\ \mbox{ and }\ \Tens_S(\overline{V}^{*\ell})_1=\overline{V}^{*\ell}=V^{*\ell}\oplus K_d^{\vee*\ell}=\triv_{d+1}(\Lambda^!)_1,\]
we have a morphism $\phi':\Tens_S(\overline{V}^{*\ell})\to\triv_{d+1}(\Lambda^!)$ of $\F$-algebras

By Proposition \ref{prop:pi-quadratic}, 
$\Pi$ is a quadratic algebra whose degree $2$ part is 
$$\Pi_2
={\displaystyle \frac{V\otimes_SV}{R}\oplus\frac{(V\otimes_SK_d^\vee)\oplus(K_d^\vee\otimes_SV)}{\delta_d'(K_{d-1}^\vee)}\oplus K_d^\vee\otimes_SK_d^\vee}$$
where we use the notation \eqref{presentation of Lambda}. Therefore $\Pi^!$ is also a quadratic algebra whose degree $2$ part is
\begin{equation*}\label{degree 2 of Pi^!}
(\Pi^!)_2={\displaystyle \frac{V^{*\ell}\otimes_SV^{*\ell}}{R^\perp}\oplus\frac{(K_d^{\vee*\ell}\otimes_SV^{*\ell})\oplus(V^{*\ell}\otimes_SK_d^{\vee*\ell})}{\delta_d'(K_{d-1}^\vee)^\perp}\oplus\frac{K_d^{\vee*\ell}\otimes_SK_d^{\vee*\ell}}{K_d^{\vee*\ell}\otimes_SK_d^{\vee*\ell}}}.
\end{equation*}
On the other hand, we have
\[\triv_{d+1}(\Lambda^!)_2=(\Lambda^!\oplus L')_2=\frac{V^{*\ell}\otimes_SV^{*\ell}}{R^\perp}\oplus K_{d-1}^{\vee*\ell}.\]
Now we compare $(\Pi^!)_2$ with $\triv_{d+1}(\Lambda^!)_2$.
To prove that $\phi'$ induces the desired morphism $\phi:\Pi^!\to\triv_{d+1}(\Lambda^!)$, it suffices to show that 
the following sequence is exact.
\begin{eqnarray}\label{exact1}
&0\to \delta_d'(K_{d-1}^\vee)^\perp\to(V^{*\ell}\otimes_SK_d^{\vee*\ell})\oplus(K_d^{\vee*\ell}\otimes_SV^{*\ell})\xrightarrow{\phi'}K_{d-1}^{\vee*\ell}
\end{eqnarray}
By our definition of the $\Lambda^{!\en}$-module structure on $L'$, the morphism $\phi'$ in \eqref{exact1} is $(\theta_d^\ell+(-1)^d\theta_d^r)^{*\ell}$. 
Since $\theta_d^\ell+(-1)^d\theta_d^r:K_{d-1}^\vee\to (K_d^\vee\otimes_SV)\oplus(V\otimes_SK_d^\vee)$ is the restriction of $\delta_d'$, the sequence \eqref{exact1} is exact.
In fact, for a morphism $\gamma:X\to Y$ of $S^\en$-modules, the sequence $0\to\gamma(X)^\perp\to Y^{*\ell}\xrightarrow{\gamma^{*\ell}}X^{*\ell}$ is clearly exact.
This completes the proof.

(b) Since $\phi$ is an isomorphism in degrees 0 and 1 by (a), 
we have that $\phi$ is surjective if and only if $\triv_{d+1}(\Lambda^!)$ is generated in degrees 0 and 1 as an algebra.  We know that the algebras $\Pi^!$ and $\triv_{d+1}(\Lambda^!)$ are generated by $V^{*\ell}\oplus K_d^{\vee*\ell}$ and $V^{*\ell}\oplus\hd_{\Lambda^{!\en}}L'$ respectively.
So $\phi$ is surjective if and only if $\phi$ gives a surjection $V^{*\ell}\oplus K_d^{\vee*\ell}\to V^{*\ell}\oplus \head_{\Lambda^{!\en}}L'$ if and only if $L'_1=\head_{\Lambda^{!\en}}L'$ if and only if $(\Lambda^{!*})_1=\head_{\Lambda^{!\en}}(\Lambda^{!*})$.
Applying $(-)^*$, this is equivalent to $(\Lambda^!)_d=\soc_{\Lambda^{!\en}}(\Lambda^!)$.

The latter assertion is immediate from (a).

(c)
The `only if' part is clear from part (b) and the fact that $\Pi^!$ is quadratic.
The `if' part follows from Lemma \ref{lem-quad012} because both algebras are quadratic and $\phi$ is an isomorphism in degrees $0$, $1$, and $2$.
\end{proof}

We have the following nice property of $\phi$.

\begin{theorem}\label{cor:nhered-surj}
If $\Lambda$ is Koszul and $d$-hereditary, then the natural morphism $\phi:\Pi^!\onto\Triv_{d+1}(\Lambda^!)$ is surjective.
\end{theorem}

To prove this, we need the following.

\begin{lemma}\label{Ext and socle}
Let $\Lambda$ be a Koszul algebra and $i\ge0$. Then $\Ext^i_{\Lambda^\en}(\Lambda,\Lambda^\en)_{-i}\cong(\soc_{\Lambda^{!\en}}(\Lambda^!))_i$.
\end{lemma}

\begin{proof}
Recall that $\Ext^i_{\Lambda^\en}(\Lambda,\Lambda^\en)$ is the cohomology of the complex
\[\Lambda\otimes_SK_{i-1}^\vee\otimes_S\Lambda\xrightarrow{\delta'_i}\Lambda\otimes_SK_{i}^\vee\otimes_S\Lambda\xrightarrow{\delta'_{i+1}}\Lambda\otimes_SK_{i+1}^\vee\otimes_S\Lambda.\]
Taking the degree $-i$ part, $\Ext^i_{\Lambda^\en}(\Lambda,\Lambda^\en)_{-i}$ is the kernel of the morphism
\[\delta'_{i+1}=\theta_{i+1}^\ell+(-1)^{i+1}\theta_{i+1}^r:K_i^\vee\to(K_{i+1}^\vee\otimes_SV)\oplus(V\otimes_SK_{i+1}^\vee)\subset\Lambda\otimes_SK_{i+1}^\vee\otimes_S\Lambda.\]
By adjunctions, $f\in K_i^\vee$ is in the kernel if and only if $V^{*r}\cdot f=0=f\cdot V^{*\ell}$.

On the other hand, we have $(\Lambda^!)_i=K_i^{*\ell}$ and $(\soc_{\Lambda^{!\en}}(\Lambda^!))_i=\{f\in K_i^{*\ell}\mid V^{*\ell}\cdot f=0=f\cdot V^{*\ell}\}$.
By Lemma \ref{duals} and its dual, the isomorphism $K_i^{*\ell}\cong K_i^\vee$ induces an isomorphism
\[(\soc_{\Lambda^{!\en}}(\Lambda^!))_i\cong\{f\in K_i^\vee\mid V^{*r}\cdot f=0=f\cdot V^{*\ell}\}=\Ext^i_{\Lambda^\en}(\Lambda,\Lambda^\en)_{-i}.\qedhere\]
\end{proof}

Now we are ready to prove Theorem \ref{cor:nhered-surj}.

\begin{proof}
Suppose $\phi$ is not surjective, so $\soc_{\Lambda^{!\en}}(\Lambda^!)\neq (\Lambda^!)_d$ holds by Theorem \ref{thm:canonical map}(b). 
By Lemma \ref{Ext and socle}, we have $\Ext^i_{\Lambda^\en}(\Lambda,\Lambda^\en)\neq0$,
a contradiction to our assumption that $\Lambda$ is $d$-hereditary.
\end{proof}

Now we look at the case $d=1$.

\begin{example}\label{A1 and A2}
Let $Q$ be a connected quiver and $\Lambda=\F Q$. Assume that $\Lambda$ is 1-hereditary, that is, $Q$ is not of type $A_1$ by our convention.

Then $\Pi^!$ is given by the double quiver $\overline{Q}$ with the following relations, where we denote by $(-)^*$ the canonical involution of $\overline{Q}$: For any arrows $\alpha$ and $\beta$ in $\overline{Q}$, $\alpha\beta=0$ if $\beta\neq\alpha^*$, and $\alpha\alpha^*=\pm\beta\beta^*$ if $\alpha$ and $\beta$ start at the same vertex.

This implies that, if $Q$ is not of type $A_2$, then $(\Pi^!)_i$ is non-zero if and only if $i=0$, $1$ or $2$.

If $Q$ is of type $A_2$, then
$\Pi^!$ is the path algebra of $[\xymatrix@C=1.5em{1\ar@<.2em>[r]&2\ar@<.2em>[l]}]$ and hence infinite dimensional, while $\Triv_2(\Lambda^!)$ is the factor algebra of $\Pi^!$ by the ideal generated by paths of length 3.
\end{example}

For other cases in $d=1$, we have the following.

\begin{theorem}\label{thm:dynkin-triv}
Let $Q$ be a connected acyclic quiver which is not of type $A_1$ and $\Lambda:=\F Q$ its path algebra.
Then the natural morphism $\phi:\Pi^!\to\Triv_2(\Lambda^!)$ is an isomorphism if and only if
$Q$ is not of type $A_2$. 
\end{theorem}

\begin{proof}
By Example \ref{A1 and A2}, we only have to show the `if' part.
Clearly $\Lambda^!$ is the factor algebra of $\F Q^{\op}$ by the ideal generated by all paths of length 2. Thus $(\Lambda^!)_i$ is non-zero only when $i=0$ or $1$, and $(\Lambda^!)_1=\soc{}_{\Lambda^{!\en}}(\Lambda^!)$ holds since $Q$ is not of type $A_1$.
By Theorem \ref{thm:canonical map}(b), we have that $\phi$ is surjective morphism which is an isomorphism in degrees $0$, $1$ and $2$.
On the other hand, $\Triv_2(\Lambda^!)_i$ is non-zero only when $i=0$, $1$, or $2$,
while $(\Pi^!)_i$ is non-zero only when $i=0$, $1$, or $2$ by Example \ref{A1 and A2}.
Thus the assertion follows.
\end{proof}

As an application of Theorem \ref{thm:dynkin-triv}, we recover a well-known result, which is mentioned in Section 5.1 of \cite{bbk} and in the introduction of \cite{hk}:

\begin{corollary}
Let $Q$ be a connected quiver which is not of type $A_1$ or $A_2$ and has bipartite orientation, and $\Lambda:=\F Q$ its path algebra. Then
$$\Pi^!\cong\Triv(\Lambda).$$
\end{corollary}
\begin{proof}This is a consequence of Theorem \ref{thm:dynkin-triv} because, when $Q$ has bipartite orientation, we have $\Lambda^!\cong \Lambda^\op$ and $\Triv(\Lambda^\op)\cong\Triv(\Lambda)$.
Moreover, as $Q$ is bipartite, the algebra automorphism $\sigma$ is inner: it is induced by a change of sign at either the sources or the sinks. Thus $\Triv_2(\Lambda)\cong\Triv(\Lambda)$.
\end{proof}

\begin{example}\label{ex-phi}
Note that our map $\phi$ is not necessarily injective nor surjective. Let $\Lambda$ be the Koszul algebra given by taking the quotient of the path algebra of the quiver
$$1\arr{\alpha}2\arr{\beta}3\arr{\gamma}4$$
by the ideal $(\alpha\beta)$.  Then $\Pi^!$ is infinite dimensional and $\Triv_3(\Lambda^!)$ is $16$-dimensional.  The kernel of $\phi$ is the infinite-dimensional space $(\Pi^!)_{\geq 4}$ and the cokernel is $2$-dimensional, generated by $\gamma\in K_1\subseteq \Triv_3(\Lambda^!)_2$ and $e_4\in K_0\subseteq \Triv_3(\Lambda^!)_3$.
\end{example}

\subsection{Type $A$ examples}
We finish this article by applying our theory to higher type A $d$-representation finite algebras \cite{iya-ct-higher,io-napr}.

Let $1\leq d<\infty$ and $2\leq s<\infty$.
Let $\overline{Q}^{(d,s)}$ denote the quiver whose vertices are $d+1$-tuples $x=(x_1,\ldots,x_{d+1})$ of nonnegative integers that sum to $s-1$, and whose arrows are
\[\alpha_{x,i}:x\to x+f_i\]
for $1\leq i\leq d+1$ whenever $x_i\geq1$, where
\[f_i=(0,\ldots,0,\stackrel{i}{-1},\stackrel{i+1}{+1},0,\ldots,0)\ \mbox{ and }\ f_{d+1}=(1,\ldots,0,\ldots,0,-1).\]
Let $Q^{(d,s)}$ be the quiver obtained by removing all arrows of the form $\alpha_{x,d+1}$ from $\overline{Q}^{(d,s)}$.
For example, the quivers $\overline{Q}^{(2,5)}$ and $Q^{(2,5)}$ are the following.
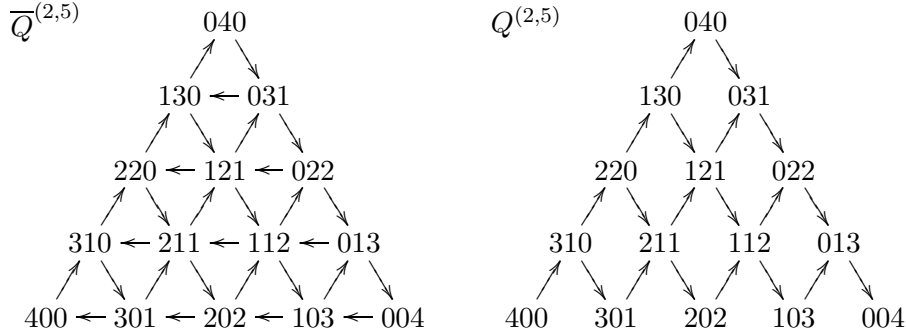
\[\begin{xy}
<0pt,0pt>;<0.5pt,0pt>:<0pt,-0.5pt>::
(0,0) *+{\overline{Q}^{(2,5)}},
(134,0) *+{040} ="0",
(100,56) *+{130} ="1",
(167,56) *+{031} ="2",
(67,112) *+{220} ="3",
(134,112) *+{121} ="4",
(200,112) *+{022} ="5",
(33,167) *+{310} ="6",
(100,167) *+{211} ="7",
(167,167) *+{112} ="8",
(234,167) *+{013} ="9",
(0,223) *+{400} ="10",
(67,223) *+{301} ="11",
(134,223) *+{202} ="12",
(200,223) *+{103} ="13",
(267,223) *+{004} ="14",
"1", {\ar"0"},
"0", {\ar"2"},
"2", {\ar"1"},
"3", {\ar"1"},
"1", {\ar"4"},
"4", {\ar"2"},
"2", {\ar"5"},
"4", {\ar"3"},
"6", {\ar"3"},
"3", {\ar"7"},
"5", {\ar"4"},
"7", {\ar"4"},
"4", {\ar"8"},
"8", {\ar"5"},
"5", {\ar"9"},
"7", {\ar"6"},
"10", {\ar"6"},
"6", {\ar"11"},
"8", {\ar"7"},
"11", {\ar"7"},
"7", {\ar"12"},
"9", {\ar"8"},
"12", {\ar"8"},
"8", {\ar"13"},
"13", {\ar"9"},
"9", {\ar"14"},
"11", {\ar"10"},
"12", {\ar"11"},
"13", {\ar"12"},
"14", {\ar"13"},
\end{xy}\ \ \ \ \ 
\begin{xy}
<0pt,0pt>;<0.5pt,0pt>:<0pt,-0.5pt>::
(0,0) *+{Q^{(2,5)}},
(134,0) *+{040} ="0",
(100,56) *+{130} ="1",
(167,56) *+{031} ="2",
(67,112) *+{220} ="3",
(134,112) *+{121} ="4",
(200,112) *+{022} ="5",
(33,167) *+{310} ="6",
(100,167) *+{211} ="7",
(167,167) *+{112} ="8",
(234,167) *+{013} ="9",
(0,223) *+{400} ="10",
(67,223) *+{301} ="11",
(134,223) *+{202} ="12",
(200,223) *+{103} ="13",
(267,223) *+{004} ="14",
"1", {\ar"0"},
"0", {\ar"2"},
"3", {\ar"1"},
"1", {\ar"4"},
"4", {\ar"2"},
"2", {\ar"5"},
"6", {\ar"3"},
"3", {\ar"7"},
"7", {\ar"4"},
"4", {\ar"8"},
"8", {\ar"5"},
"5", {\ar"9"},
"10", {\ar"6"},
"6", {\ar"11"},
"11", {\ar"7"},
"7", {\ar"12"},
"12", {\ar"8"},
"8", {\ar"13"},
"13", {\ar"9"},
"9", {\ar"14"},
\end{xy}
\]

\medskip
Let $I^{(d,s)}$ denote the ideal of $\F Q^{(d,s)}$ generated by elements:
\begin{align*}
\alpha_{x,i}\alpha_{x+f_i,j}&=\alpha_{x,j}\alpha_{x+f_j,i}&&\text{ if $x_i,x_j\geq1$;}\\
\alpha_{x,i}\alpha_{x+f_i,i+1}&=0&&\text{ if $x_i\geq1$ and $x_{i+1}=0$,}
\end{align*}
where $x\in Q^{(d,s)}_0$ and $1\le i<j\le d$.

For a field $\F$, let
\[\Lambda^{(d,s)}=\F Q^{(d,s)}/I^{(d,s)}.\]
Then $\Lambda^{(d,s)}$ is $d$-RF \cite[Theorems 1.18, 6.12]{iya-ct-higher}.
Also, as $I^{(d,s)}$ is a homogeneous ideal with respect to the path length grading on $\F Q^{(d,s)}$, $\Lambda^{(d,s)}$ inherits this grading.

The following notation will be useful: for a vertex $x$ in $Q^{(d,s)}$, let $e_x$ denote the idempotent of $\Lambda$ corresponding to the vertex $x$, and let
\[\alpha_i=\sum \alpha_{x,i}.\]
Then the relations in $\Lambda^{(d,s)}$ can be rewritten as:
\[ e_x(\alpha_i\alpha_j-\alpha_j\alpha_i)=0 \text{ for all vertices $x$ and all $i\neq j$}. \]

We have a natural morphism $\phi:\Pi^!\to\Triv_{d+1}(\Lambda^!)$.  We know by Theorem \ref{thm:canonical map} and Corollary \ref{cor:nhered-surj} that $\phi$ is always surjective.  If $s\geq3$, then it is shown in \cite[Section 3]{g-zz} that $\phi$ is an isomorphism.

We will make use of the following result:
\begin{proposition}[{\cite[Proposition 3.4]{g-zz}}]
$\Lambda$ is a Koszul algebra.
\end{proposition}

\begin{lemma}\label{lem:K_d-type-a}
 The space $K_d$ has an $S^\en$-module basis $\{k_x\st x\in Q_0, x_1\neq0\}$, where
 $$k_x=e_x \sum_{\sigma\in S_d}({\rm sgn}\,\sigma)\alpha_{\sigma(1)}\alpha_{\sigma(2)}\cdots\alpha_{\sigma(d)}.$$
\end{lemma}
\begin{proof}
Fix $0\leq r\leq d-2$.  First we show that $k_x\in V^r R V^{d-r-2}$.  For any vertex $y$ and any $i\neq j$ we have $e_y(\alpha_i\alpha_j-\alpha_j\alpha_i)\in R$.  Thus, for any indices $i_1,\ldots,i_{d-2}$ such that $\{i,j, i_1,\ldots,i_{d-2}\}=\{1,2,\ldots,d\}$ we have $e_x\alpha_{i_1}\ldots\alpha_{i_r}(\alpha_i\alpha_j-\alpha_j\alpha_i)\alpha_{r+1}\ldots\alpha_{d-2}\in V^r R V^{d-r-2}$.  Summing over all such sets $\{i_1,\ldots,i_{d-2}\}$, with sign, we get that $k_x\in V^r R V^{d-r-2}$.  But this did not depend on $r$, so we have $k_x\in K_d=\bigcap_{r=0}^{d-2}V^r R V^{d-r-2}$.

Conversely, consider an element $k\in K_d$.  Without loss of generality, $k=e_xk$ for some vertex $x$.  No summand of $k$ can be of the form $p\alpha_i\alpha_iq$ with $p\in V^r$ and $q\in V^{d-r-2}$, or else $k\notin V^r R V^{d-r-2}$.  So we must have 
\[ k=e_x \sum_{\sigma\in S_d}\lambda_\sigma\alpha_{\sigma(1)}\alpha_{\sigma(2)}\cdots\alpha_{\sigma(d)} \]
for some scalars $\lambda_\sigma\in \F$.  But this can only be in $RV^{d-2}$ if $\lambda_\sigma+\lambda_{(12)\sigma}=0$.  Similarly, we have $\lambda_\sigma+\lambda_{(i,i+1)\sigma}=0$ for all $1\leq i<d$.  Thus ${\rm sgn}\,\sigma={\rm sgn}\,\tau$ implies $\lambda_\sigma=\lambda_\tau$, and ${\rm sgn}\,\sigma=-{\rm sgn}\,\tau$ implies $\lambda_\sigma=-\lambda_\tau$.  So $k$ is a scalar multiple of $k_x$.
\end{proof}

Let $\overline{I}^{(d,s)}$ denote the ideal of $\F\overline{Q}^{(d,s)}$ generated by elements: 
\begin{align*}
\alpha_{x,i}\alpha_{x+f_i,j}&=\alpha_{x,j}\alpha_{x+f_j,i}&&\text{ if $x_i,x_j\geq1$;}\\
\alpha_{x,i}\alpha_{x+f_i,i+1}&=0&&\text{ if $x_i\geq1$ and $x_{i+1}=0$,}
\end{align*}
where $x\in\overline{Q}^{(d,s)}_0$ and $1\le i<j\le d+1$. Here, $\alpha_{x+f_{d+1},d+2}$ should be interpreted as $\alpha_{x+f_{d+1},1}$.
As an application of our results in this paper, we give the following description of the higher preprojective algebra of $\Lambda$, which recovers the quiver with relations in 
\cite[Definition 5.1, Proposition 5.48]{io-napr}.

\begin{theorem}
Let $\Pi=\Pi(\Lambda)$.
The quiver $\overline{Q}$ of $\Pi$ is $\overline{Q}^{(d,s)}$,
and we have an isomorphism
\[\Pi\cong\F\overline{Q}^{(d,s)}/\overline{I}^{(d,s)}.\]
\end{theorem}

\begin{proof}
The former statement follows from Proposition \ref{prop:piquiv}. We prove the latter one.
From Lemma \ref{lem:K_d-type-a} we obtain the superpotential
$$W=\sum_{\sigma\in S_d}({\rm sgn}\,\sigma)\alpha_{\sigma(1)}\alpha_{\sigma(2)}\cdots\alpha_{\sigma(d)}\alpha_{d+1}$$
for $\overline{Q}$.
By differentiating this superpotential with respect to all paths of length $d-1$ in $\overline{Q}$, we have the isomorphism.
\end{proof}

We now apply Theorem \ref{thm:kos-alm-kos} to obtain a large family of pairs of almost Koszul algebras. This statement generalizes \cite[Corollary 4.3]{bbk} for type $A$ quivers.  It appears to be the first construction of $(p,q)$-Koszul algebras for all $p,q\geq2$.
\begin{proposition}
If $s\geq3$ and $n\geq1$, then $\Pi$ and $\Pi^!$ are an almost Koszul pair: $\Pi$ is $(s-1,d+1)$-Koszul and $\Pi^!$ is $(d+1,s-1)$-Koszul.
\end{proposition}

\begin{proof}
Theorem \ref{thm:kos-alm-kos} tells us that $\Pi$ is $(p,d+1)$-Koszul if $\Pi$ is concentrated in degrees $0$ to $p$, and \cite[Proposition 3.11]{bbk} tells us that the quadratic dual of a $(p,q)$-Koszul ring with $p,q\geq2$ is a $(q,p)$-Koszul ring.  So we just need to show that $\Pi$ is concentrated in degrees $0$ to $s-1$.

We use Mart\'inez-Villa's result that all projective modules for a $\Z$-graded self-injective algebra have the same Loewy length \cite[Theorem 3.3]{mv-selfinj}.  Thus we only need to show that there is a projective $\Pi$-module concentrated in degrees $0$ to $s-1$.  Consider the left projective $\Pi$-module $\Pi e_{(s-1,0,\ldots,0)}$ associated to the vertex $x=(s-1,0,\ldots,0)$.  First we claim that all paths starting at $x$ are of the form $e_x\alpha_1^d$.  To see this, note that the arrows in $\overline{Q}$ ensure that every path not of this form starting at $x$ must begin $e_x\alpha_1^m\alpha_2$ for some $m\geq1$.  But then the commutation relations in $\Pi$ show that $e_x\alpha_1^m\alpha_2=e_x\alpha_1\alpha_2\alpha_1^{m-1}$.  But $e_x\alpha_1\alpha_2=0$.  Next we note that $e_x\alpha_1^d$ is nonzero for $0\leq d\leq s-1$ and is zero for $d\geq s$.  So $\Pi e_{(s-1,0,\ldots,0)}$ is nonzero precisely in degrees $0$ to $s-1$. 
\end{proof}

\appendix
\section{On global dimension of $\Z$-graded rings}

The aim of this appendix is to remark that several possible definitions of global dimensions of $\Z$-graded rings coincide.

For a ring $A$, we denote by $A\mlMod$ the abelian category of all left $A$-modules. For a $\Z$-graded ring $A=\bigoplus_{i\in\Z}A_i$, we denote by $A\mgrlMod$ the abelian category of all $\Z$-graded left $A$-modules. We denote by $\lModm A$ and $\grlModm A$ the right versions. If $A_i=0$ for all $i<0$, then by \cite[I.7.8]{NV} (see also 
\cite[7.6.18(ii)]{MR}) and \cite[2.4.8]{NV2}, we have
\begin{eqnarray}\label{lgldim 3}
&\gl(A\mlMod)=\gl(A\mgrlMod)=\sup\{\pd_AX\mid X\in A\mgrlMod\ \text{: cyclic}\},&\\ \label{rgldim 3}
&\gl(\lModm A)=\gl(\grlModm A)=\sup\{\pd X_A\mid X\in \grlModm A\ \text{: cyclic}\}.&
\end{eqnarray}
The aim of this section is to prove the following general observation.

\begin{theorem}\label{global dimension}
Let $A=\bigoplus_{i\ge0}A_i$ be a $\Z$-graded ring. If $A_0$ is artinian, then
\begin{eqnarray*}
&\gl(A\mlMod)=\sup\{\pd_AX\mid X\in A\mgrlMod\ \text{: simple}\}\\
=&\gl(\lModm A)=\sup\{\pd X_A\mid X\in \grlModm A\ \text{: simple}\}.
\end{eqnarray*}
\end{theorem}

It is well-known that \eqref{lgldim 3} fails if we drop the condition $A_i=0$ for $i<0$, e.g.\ $A=k[x,x^{-1}]$ with field $A_0=k$ and $\deg x=1$. Also Theorem \ref{global dimension} fails if $A_0$ is not artinian, e.g.\ $A=A_0\oplus A_1=\Z\oplus\Q x$ is a subring of $\Q[x]/(x^2)$.

Theorem \ref{global dimension} follows immediately from \eqref{lgldim 3}, \eqref{rgldim 3} and the following observation.

\begin{lemma}\label{upper bound by simple}
Let $A=\bigoplus_{i\ge0}A_i$ be a $\Z$-graded ring. If $A_0$ is left artinian, then 
\[\gl(A\mgrlMod)\le\sup\{\pd X_A\mid X\in \grlModm A\ \mbox{: simple}\}.\]
\end{lemma}

To prove this, we need a preparation.  For $i\in\Z$, let $\PP^i$ be the full subcategory of $A\mgrlMod$ consisting of arbitrary direct sums of modules of the form $Ae(-i)$ for some idempotents $e\in A_0$. Let $\PP$ be the full subcategory of $A\mgrlMod$ consisting of arbitrary direct sums of objects from $\PP^i$ for all $i\in\Z$.

\begin{lemma}\label{existence of minimal resolution}
Let $A=\bigoplus_{i\ge0}A_i$ be a $\Z$-graded ring such that $A_0$ is left artinian, and $J:=(\rad A_0)\oplus(\bigoplus_{i\ge1}A_i)$. For any $X\in A\mgrlMod$ such that $X_j=0$ for $j\ll0$, there exists a surjective morphism $f:P\to X$ in $A\mgrlMod$ with $P\in\PP$ such that $\Ker f\subset JP$ and $(\Ker f)_j=0$ for $j\ll0$.
\end{lemma}

\begin{proof}
Without loss of generality, we can assume $X_j=0$ for all $j<0$.
We take a morphism $f^0:P^0\to X$ in $A\mgrlMod$ with $P^0\in\PP^0$ such that $(f^0)_0:(P^0)_0\to X_0$ is a projective cover of the $A_0$-module $X_0$.
Assume that $f^j:P^j\to X$ with $P^j\in\PP^j$ are constructed for $0\le j<i$ such that
\[f^{[0,i)}:=(f^0,\ldots,f^{i-1}):P^{[0,i)}:=\bigoplus_{0\le j<i}P^j\to X\]
satisfies $(\Cok f^{[0,i)})_j=0$ for all $j<i$. 
We take a morphism $g^i:P^i\to\Cok f^{[0,i)}$ in $A\mgrlMod$ with $P^i\in\PP^i$ such that $(g^i)_i:(P^i)_i\to(\Cok f^{[0,i)})_i$ is a projective cover of the $A_0$-module $(\Cok f^{[0,i)})_i$.
We lift $g^i$ to $f^i:P^i\to X$. Then $f^j$ with $0\le j\le i$ satisfy the same assumption.

Now we show that the morphism $f:=(f^j)_{j\le0}:P:=\bigoplus_{0\le j}P^j\to X$ satisfies the desired properties. Clearly $f$ is surjective, and $P\in\PP$ and $(\Ker f)_j=0$ holds for all $j<0$. It remains to show $\Ker f\subset JP$.
Take any $x=(x^j)_{j\ge0}\in\Ker f$ with $x^j\in P^j$. Then the composition $P\xrightarrow{f}X\to\Cok f^{[0,i)}$ sends $(x^j)_{j\ge i}$ to zero, and $(x^j)_{j>i}$ to an element in $(\Cok f^{[0,i)})_{>i}$. Thus $g^i(x^i)$ belongs to $(\Cok f^{[0,i)})_{>i}$.
Since $(g^i)_i:(P^i)_i\to(\Cok f^{[0,i)})_i$ is a projective cover, $x^i\in JP^i$ holds, as desired.
\end{proof}

Now we are ready to prove Lemma \ref{upper bound by simple}.

\begin{proof}[Proof of Lemma \ref{upper bound by simple}]
Since $A_0$ is left artinian, $A/J=A_0/\rad A_0$ is a semisimple ring and $A/J$ is a semisimple right $A$-module. Let $\ell=\pd(A/J)_A$. Thanks to \eqref{lgldim 3}, it suffices to show that $\pd_AX\le\ell$ holds for any $X\in A\mgrlMod$ which is cyclic. Since $X_i=0$ holds for $i\ll0$, by applying Lemma \ref{existence of minimal resolution} to $X$ and its syzygies repeatedly, we obtain an exact sequence
\[\cdots\xrightarrow{f^3}Q^2\xrightarrow{f^2}Q^1\xrightarrow{f^1}Q^0\to X\to0\]
such that $Q^i\in\PP$ and $f^i(Q^i)\subset JQ^{i-1}$ for all $i$. Since $(A/J)\otimes f^i=0$ for all $i>0$, we have
\[Q^{\ell+1}/JQ^{\ell+1}=(A/J)\otimes_AQ^{\ell+1}=\Tor^A_{\ell+1}(A/J,X)=0.\]
Thus $Q^{\ell+1}=0$ and hence $\pd_AX\le\ell$.
\end{proof}

%
%

\end{document}